\documentclass[aap,preprint]{imsart}
\RequirePackage[OT1]{fontenc}
\RequirePackage{amsthm,amsmath}
\RequirePackage[colorlinks,citecolor=blue,urlcolor=blue]{hyperref}
\usepackage{amsthm}
\usepackage{amsmath}
\usepackage{amssymb,amsfonts}    
\usepackage{dsfont}
\usepackage{graphicx}
\arxiv{arXiv:1712.05743}

    \usepackage[toc,page]{appendix}
\usepackage{newpxtext} 
\usepackage{bbm}								
\theoremstyle{plain}
\newtheorem{theorem}{Theorem}[section]
\newtheorem{lemma}[theorem]{Lemma}
\newtheorem{proposition}[theorem]{Proposition}
\newtheorem{corollary}[theorem]{Corollary}
\newtheorem{remark}[theorem]{Remark}
\usepackage{color}

\newcommand{\inv}{^{-1}}

\newcommand{\E}{\mathbb{E}}

\newcommand{\ind}{\mathds{1}}

\begin{document}

\begin{frontmatter}
\title{Stein's Method for Stationary Distributions of  Markov Chains and Application to Ising Models}
\runtitle{Stein's Method for Stationary Distributions of Markov Chains}


\begin{aug}
\author{\fnms{Guy} \snm{Bresler}\ead[label=e1]{guy@mit.edu}}
\and
\author{\fnms{Dheeraj} \snm{Nagaraj}\thanksref{t1}\ead[label=e2]{dheeraj@mit.edu}}

\thankstext{t1}{This work was supported in part by grants NSF CCF-1565516, ONR N00014-17-1-2147, and DARPA W911NF-16-1-0551.}
\runauthor{Bresler and Nagaraj}

\affiliation{Massachusetts Institute of Technology }

\address{Dept of EECS\\
Mass. Inst. of Tech. \\
Cambridge, MA 02139 \\
\printead{e1}}

\address{Dept of EECS\\
Mass. Inst. of Tech. \\
Cambridge, MA 02139 \\
\printead{e2}}
\end{aug}

\begin{abstract}
 We develop a new technique, based on Stein's method, for comparing two stationary distributions of irreducible Markov Chains whose update rules are close in a certain sense. We apply this technique to compare Ising models on $d$-regular expander graphs to the Curie-Weiss model (complete graph) in terms of pairwise correlations and more generally $k$th order moments. Concretely, we show that $d$-regular Ramanujan graphs approximate the $k$th order moments of the Curie-Weiss model to within average error $k/\sqrt{d}$ (averaged over size $k$ subsets), \emph{independent of graph size}. The result applies even in the low-temperature regime; we also derive simpler approximation results for functionals of Ising models that hold only at high temperatures.
\end{abstract}


\end{frontmatter}

\section{Introduction}

Markov random fields (MRFs) are widely used in a variety of applications as models for high-dimensional data. The primary reasons are interpretability of the model, whereby edges between variables indicate direct interaction, and efficiency of carrying out inference tasks such as computation of marginals or posteriors. Both of these objectives are helped by \emph{sparsity} of the model: edges can more easily be assigned meaning if there are few of them, and each update step in inference algorithms such as belief propagation or Gibbs sampler require computation depending on the degrees of the nodes. (While each update or iteration can be carried out more efficiently in a sparse model, it is not clear how to compare the number of iterations needed. In general, carrying out inference tasks is computationally hard even in bounded-degree models~\cite{sly2010computational}.) 

This paper takes a first step towards understanding what properties of an MRF with many edges can be captured by a model with far fewer edges. We focus on the Ising model, the canonical binary pairwise graphical model. Originally introduced by statistical physicists to study phase transitions in magnetic materials~\cite{ising1925beitrag,brush1967history}, these distributions capture rich dependence structure and are widely used in a variety of applications including for modeling images, neural networks, voting data and biological networks~\cite{banerjee2008model,greig1989exact, schneidman2006weak}. 
The Ising model assigns to each configuration $x\in \{-1,+1\}^n$ probability
$$
p(x)= \frac1Z \exp\big({\tfrac{1}{2}x^\intercal J x}\big)\,,
$$
where $J\in \mathbb{R}^{n\times n}$ is a symmetric matrix of interactions and the partition function $Z$ normalizes the distribution. The support of the interaction matrix $J$ is represented by a graph $G_J=([n], E_J)$ with $\{i,j\}\in E_J$ if and only if $J_{ij}\neq 0$. The Curie-Weiss model at `inverse temperature' $\beta$ is the Ising model on the complete graph with all entries of the interaction matrix $J$ equal to~$\frac{\beta}{n}$.

Sparsification of graphs~\cite{spielman2011spectral,batson2013spectral} has in recent years had a large impact in theoretical computer science. The notion of approximation in that literature is spectral: given a graph with Laplacian $L$, the objective is to find a sparser graph with Laplacian $M$ such that $x^\intercal L x \approx x^\intercal M x$ for all $x$. 
The Ising model sufficient statistic, $x^\intercal J x$, is thus approximately preserved by spectral graph sparsification, but it is not clear how this translates to any sort of notion of nearness of the \emph{distributions} of corresponding Ising models, because of their inherent non-linearity. 

In this paper we initiate the study of the interplay between spectral approximation of graphs and Ising models by showing that low-order moments of the Curie-Weiss model (Ising model on the complete graph with uniform edge-weights) are accurately represented by expander graphs (which are spectral approximations of the complete graph). As discussed in~\cite{bresler2016learning},
low-order moments capture the probabilistic content of a model relevant to the machine learning task of making predictions based on partial observations. Our main result shows that $k$th order moments in the Curie-Weiss model are approximated to average accuracy $k/\sqrt{d}$ by $d$-regular approximate Ramanujan graphs (and more generally to average accuracy $k\epsilon$ by $\epsilon$-expander graphs). 

\begin{theorem}[Informal version of Theorem~\ref{main_application}]
The $k$th order moments of the Curie-Weiss model on $n$ nodes with inverse temperature $\beta$ are approximated to within average error $kC(\beta)/\sqrt{d}$ by an Ising model on a $d$-regular approximate Ramanujan graph. 
\end{theorem}

We note that random regular graphs are known to be approximately Ramanujan with high probability.
The proof is based on a coupling argument together with the  abstract comparison technique developed in this paper; in order to deal with the low-temperature regime where Glauber dynamics mixes slowly, we use the restricted dynamics studied in~\cite{levin2010glauber}. A much weaker bound can be obtained via the Gibbs variational principle, and we outline that method in Section~\ref{s:naive}.

The techniques developed in the paper are likely to be of independent interest because of their applicability to other models, but we do not pursue that here. We frame our basic goal as that of comparing the expectations of a Lipschitz function under two distributions, and to that end we prove a bound in terms of nearness of \emph{Markov kernels} with desired stationary distributions. Specifically, our main abstract result, Theorem~\ref{main_theorem}, is stated in terms of the Glauber dynamics for the two distributions. We prove this theorem in Section~\ref{s:abstractResult}.
The technique is based on Stein's method, which we review briefly in Section~\ref{s:prelim} along with relevant background on the Glauber dynamics and the Poisson equation. For any distribution $\mu(\cdot)$ over $\{-1,1\}^n$, we denote by $\mu_i(\cdot|x^{(\sim i)})$ the conditional distribution of the $i$th coordinate when the value of every other coordinate (denoted by $x^{(\sim i)}$) is fixed.

\begin{theorem}[Short version of Theorem~\ref{main_theorem}]\label{t:main_short}
	Let $\mu$ and $\nu$ be probability measures on $\Omega = \{-1,+1\}^n$. Let $P$ be the kernel of Glauber dynamics with respect to $\mu$. Let $f : \Omega \to \mathbb{R}$ be any function and let $h:\{-1,1\}^n\to \mathbb{R}$ be a solution to the Poisson equation $h - Ph = f - \mathbb{E}_\mu f$. Then
	\begin{equation}
	|\mathbb{E}_\mu f - \mathbb{E}_\nu f| \leq \mathbb{E}_{\nu}\Big(\frac{1}{n}\sum_{i=1}^{n} |\Delta_i(h)| \cdot \|\mu_i(\cdot|x^{(\sim i)})-\nu_i(\cdot|x^{(\sim i)})\|_{\mathsf{TV}}\Big)\,,
	\end{equation}
where $\Delta_i(h)$ is the discrete derivative of $h$ along the coordinate $i$.
\end{theorem}
	If $P$ is contractive and $f$ is Lipschitz, then we get a simplified bound, given in Theorem~\ref{main_theorem}. Aside from applying the technique to prove Theorem~\ref{main_application} on approximation of Ising moments, we state a result in Subsection~\ref{ss:Dobrushin} comparing functionals of an Ising model with a perturbed Ising model when one of them has sufficiently weak interactions (specifically, we require a condition similar to, though slightly weaker than, Dobrushin's uniqueness condition). 

\begin{remark}
The same result as stated in Theorem~\ref{t:main_short}, with a similar proof, was discovered independently in~\cite{reinert2017}. Their main application is to compare exponential random graphs with Erd\H{o}s-R\'enyi random graphs, whereas we use it to compare Ising models to the Curie-Weiss model. For added transparency we have coordinated the submissions of our two papers.
\end{remark}

We briefly outline the rest of the paper. Section 2 reviews Stein's method, the Poisson equation, Glauber dynamics, and motivates our technique. Section 3 states and proves the main abstract result.  Section~\ref{s:IsingApproximation} contains the application to Ising models with weak interactions and  our result on approximation of moments of the Curie-Weiss model by those of Ising models on expanders. The proof of the former is in Section~\ref{s:contractProof} and of the latter in Sections~\ref{s:ProofIdeas} and~\ref{s:IsingProof}.  
  
We remark that several papers consider the problem of \emph{testing} various properties of an Ising model from samples, such as whether the variables are jointly independent, equal to a known Ising model, etc. \cite{daskalakis2016testing,daskalakis2017concentration,gheissari2017concentration}. The problem of testing between dense and sparse Ising models is studied in \cite{bresler2018optimal}. 

\section{Preliminaries}
\label{s:prelim}
\subsection{Stein's Method}
Stein's method was first introduced by Charles Stein in his famous paper \cite{stein1972bound} to prove distributional convergence of sums of random variables to a normal random variable even in the presence of dependence. The method gives explicit Berry-Esseen-type bounds for various probability metrics. The method has since been used to prove distributional convergence to a number of distributions including the Poisson distribution \cite{chen1975poisson}, the exponential distribution \cite{chatterjee2011exponential,Fulman2013} and $\beta$ distribution \cite{goldstein2013stein,dobler2015stein}. See~\cite{ross2011fundamentals} for a survey of Stein's method; we give a brief sketch. 

Consider a sequence of random variables $Y_n$ and a random variable $X$. Stein's method is a way prove distributional convergence of $Y_n$ to $X$ with explicit upper bounds on an appropriate probability metric (Kolmogorov-Smirnov, total variation, Wasserstein, etc.). This involves the following steps:
\begin{enumerate}
\item
Find a characterizing operator $\mathcal{A}$ for the distribution of $X$, which maps functions $h$ over the state space of $X$ to give another function $\mathcal{A}h$ such that $$\mathbb{E}[ \mathcal{A}(h)(X)] = 0 \,.$$ 
Additionally, if $\mathbb{E}\mathcal{A}(h)(Y) = 0$ for a large enough class of functions $h$,
then $Y \stackrel{d}{=} X$.  Therefore the operator $\mathcal{A}$ is called a `characterizing operator'.
\item
For an appropriate class of functions $\mathcal{F}$ (depending on the desired probability metric), one solves the Stein equation
$$\mathcal{A}h = f - \mathbb{E}f(X)$$ for arbitrary $f \in \mathcal{F}$.
\item
By bounding $|\mathbb{E}f(Y_n) - \mathbb{E}f(X)|$ in terms of $\mathbb{E}\mathcal{A}(h)(Y_n)$, which is shown to be tending to zero, it follows that $Y_n \stackrel{d}{\to} X$.
\end{enumerate}

 The procedure above is often carried out via the method of exchangeable pairs (as done in Stein's original paper \cite{stein1972bound}; see also the survey by \cite{ross2011fundamentals} for details). An exchangeable pair $(Y_n, Y_n^{\prime})$ is constructed such that $Y_n^{\prime}$ is a small perturbation from $Y_n$ (which can be a step in some reversible Markov chain). Bounding the distance between $X$ and $Y_n$ then typically reduces to bounding how far $Y_n^{\prime}$ is from $Y_n$ in expectation.
Since reversible Markov chains naturally give characterizing operators as well as `small perturbations', we formulate our problem along these lines.

\subsection{Markov Chains and the Poisson Equation}
In this paper, we only deal with finite state reversible and irreducible Markov Chains. Basic definitions and methods can be found in \cite{levin2009markov} and \cite{aldous2000reversible}. Henceforth, we use the notation in \cite{levin2009markov} for our exposition on Markov chains.  Let $P$ be an irreducible Markov kernel and $\mu$ be its unique stationary distribution. We denote by $\mathbb{E}_\mu$  the expectation with respect to the measure $\mu$. It will be convenient to use functional analytic notation in tandem with probability theoretic notation for expectation, for instance replacing $\mathbb{E}g(X)$ for a variable $X\sim 
\mu$ by $\mathbb{E}_\mu g$.

Given a function $f : \Omega \to \mathbb{R}$, we consider the following equation called the Poisson equation: 
\begin{equation}
h - Ph = f - \mathbb{E}_{\mu}f \,.
\label{poisson_equation}
\end{equation}
 By definition of stationary distribution, $\mathbb{E}_{\mu}(h -Ph) = 0$. By uniqueness of the stationary distribution, it is clear that for any probability distribution $\eta$ over the same state space as $\mu$, $\mathbb{E}_{\eta}(h -Ph) = 0$ for all $h$ only if $\mu = \eta$. Therefore, we will use Equation~\eqref{poisson_equation} as the Stein equation and the operator $I-P$ as the characterizing operator for $\mu$. The Poisson equation was used in \cite{chatterjee2005concentration} to show sub-Gaussian concentration of Lipschitz functions of weakly dependent random variables using a variant of Stein's method.

For the finite state, irreducible Markov chains we consider, solutions can be easily shown to exist in the following way:
The Markov kernel $P$ can be written as a finite stochastic matrix and functions over the state space as column vectors. We denote the pseudo-inverse of the matrix $I-P$ by $(I-P)^\dagger$, and one can verify that $h = (I-P)^\dagger(f-\mathbb{E}_{\mu}f)$ is a solution to~\eqref{poisson_equation}. The solution to the Poisson equation is not unique: if $h(x)$ is a solution, then so is $h(x)+a$ for any $a \in \mathbb{R}$. We refer to the review article by Makowski and Schwartz in \cite{feinberg2012handbook} and references therein for material on solution to the Poisson equation on finite state spaces.

We call the solution $h$ given in the following lemma the \emph{principal solution} of the Poisson equation. See \cite{feinberg2012handbook} for the proof.
\begin{lemma}\label{principal_solution}
  Let the sequence of random variables $(X_i)_{i=0}^{\infty}$ be a Markov chain with transition kernel $P$. Suppose that $P$ is a finite state irreducible Markov kernel with stationary distribution~$\mu$. Then the Poisson equation~\eqref{poisson_equation} has the following solution:
$$h(x) = \sum_{t=0}^{\infty} \mathbb{E}\left[f(X_t) - \mathbb{E}_{\mu}f \mid X_0 = x\right] \ .$$
\end{lemma}

\subsection{Glauber Dynamics and Contracting Markov Chains}\label{ss:glauber}
Given $x \in \Omega= \{-1,+1\}^n$, let $x^{(\sim  i)}$ be the values of $x$ except at the $i$th coordinate. For any probability measure $p(\cdot)$ over $\Omega$ such that $p(x^{(\sim i)}) > 0$, we let $p_i(\cdot|x^{(\sim i)})$ denote the conditional distribution of the $i$th coordinate given the rest to be $x^{(\sim i)}$. We also denote by $x^{(i,+)}$ (and $x^{(i,-)}$) the vectors obtained by setting the $i$th coordinate of $x$ to be $1$ (and $-1$). For any real-valued function $f$ over $\Omega$, denote the discrete derivative over the $i$th coordinate by $\Delta_i(f) := f\left(x^{(i,+)}\right) - f\left(x^{(i,-)}\right)$.

Given a probability measure $p$ over a product space $\mathcal{X}^n$, the \emph{Glauber Dynamics} generated by $p(\cdot)$ is the following Markov chain:
\begin{enumerate}
\item Given current state $X \in \mathcal{X}^n$, pick $I \in [n]$ uniformly and independently. 
\item Pick the new state $X^{\prime}$ such that $(X^{\prime})^i = X^i$ for all $i\neq I$. 
\item  The $I$th coordinate $(X^{\prime})^I$ is obtained by resampling according to the conditional distribution $p_I(\cdot|X^{(\sim I)})$.
\end{enumerate}
All the Glauber dynamics chains considered  in this paper are irreducible, aperiodic, reversible and have the generating distribution as the unique stationary distribution.


Denote the Hamming distance by $d_H(x,y) = \sum_{i=1}^n \ind_{x^i \neq y^i}$. Consider two Markov chains $(X_t)$ and $(Y_t)$ evolving according to the same Markov transition kernel $P$ and with different initial distributions. Let $\alpha \in [0,1)$. We call the Markov kernel $P$ $\alpha$-contractive (with respect to the Hamming metric) if there exists a coupling between the chains such that $\mathbb{E}[d_H(X_{t},Y_{t})|X_0 =x,Y_0 =y] \leq \alpha^t d_H(x,y)$ for all $t \in \mathbb{N}$.

\section{The Abstract Result}\label{s:abstractResult}

%

Given two real-valued random variables $W_1$ and $W_2$, the $1$-Wasserstein distance between their distributions is defined as
$$d_W(W_1,W_2) = \sup_{g \in 1\text{-Lip}} \mathbb{E}g(W_1) - \mathbb{E}g(W_2)\,.$$
Here the supremum is over 1-Lipschitz functions $g:\mathbb{R} \to \mathbb{R}$.

\begin{theorem}[The abstract result]\label{main_theorem}
Let $\mu$ and $\nu$ be probability measures on $\Omega = \{-1,+1\}^n$ with Glauber dynamics kernels $P$ and $Q$, respectively. Additionally, let $P$ be irreducible. Let $f : \Omega \to \mathbb{R}$ be any function and let $h$ be a solution to the Poisson equation~\eqref{poisson_equation}. Then
\begin{equation}
|\mathbb{E}_\mu f - \mathbb{E}_\nu f| \leq 
\mathbb{E}_{\nu}\Big(\frac{1}{n}\sum_{i=1}^{n} |\Delta_i(h)| \cdot \|\mu_i(\cdot|x^{(\sim i)})-\nu_i(\cdot|x^{(\sim i)})\|_{\mathsf{TV}}\Big)\,.
\label{main_bound}
\end{equation}
Furthermore, if $P$ is $\alpha$-contractive and the function $f$ is $L$-Lipschitz with respect to the Hamming metric, then
\begin{equation}
|\mathbb{E}_\mu f - \mathbb{E}_\nu f| \leq \frac{L}{(1-\alpha)}\mathbb{E}_{\nu}\Big(\frac{1}{n}\sum_{i=1}^{n}  \|\mu_i(\cdot|x^{(\sim i)})-\nu_i(\cdot|x^{(\sim i)})\|_{\mathsf{TV}}\Big)\,.
\label{contractive_chain_bound}
\end{equation}
If $Z_\mu \sim \mu$ and $Z_\nu \sim \nu$, then
\begin{equation}
d_W(f(Z_\mu),f(Z_\nu)) \leq \frac{L}{(1-\alpha)}\mathbb{E}_{\nu}\Big(\frac{1}{n}\sum_{i=1}^{n}  \|\mu_i(\cdot|x^{(\sim i)})-\nu_i(\cdot|x^{(\sim i)})\|_{\mathsf{TV}}\Big)\,.
\label{contractive_chain_wasserstein_bound}
\end{equation}
\end{theorem}

\begin{proof}
To begin, since $\nu$ is stationary for $Q$, $\mathbb{E}_\nu h = \mathbb{E}_\nu Qh$. Taking expectation with respect to $\nu$ in~\eqref{poisson_equation}, we get
\begin{equation}
\mathbb{E}_\nu (Q-P)h = \mathbb{E}_\nu f - \mathbb{E}_\mu f\,.
\label{poisson_expectation}
\end{equation}
By definition of the Glauber dynamics, 
\begin{align}
(Q-P)h &= \frac{1}{n}\sum_{i=1}^{n} \Big(h(x^{(i,+)})\nu_i(1|x^{(\sim i)}) + h(x^{(i,-)})\nu_i(-1|x^{(\sim i)})\nonumber \\ &\qquad \qquad \quad - h(x^{(i,+)})\mu_i(1|x^{(\sim i)}) - h(x^{(i,-)})\mu_i(-1|x^{(\sim i)})\Big) \nonumber\\
&= \frac{1}{n}\sum_{i=1}^{n} \Delta_i(h)\big(\nu_i(1|x^{(\sim i)})  - \mu_i(1|x^{(\sim i)}) \big) \,.
\label{tensorisation_glauber}
\end{align}
Combining \eqref{poisson_expectation}
and \eqref{tensorisation_glauber}, along with the triangle inequality, yields~\eqref{main_bound}.

To prove~\eqref{contractive_chain_bound}, it is sufficient to show that if $f$ is $L$-Lipschitz and $P$ is $\alpha$-contractive, then $\Delta_i(h) \leq \frac{L}{1-\alpha}$. This we achieve using Lemma~\ref{principal_solution}. Let $(X_t)$, $(Y_t)$ be Markov chains evolving with respect to the kernel $P$, coupled such that they are $\alpha$-contractive. Then,
\begin{align*}
|\Delta_i(h)(x)| &=  \biggr|\sum_{t=0}^{\infty} \mathbb{E}\left[f(X_t) - f(Y_t) \mid X_0 = x^{(i,+)}, Y_0 = x^{(i,-)}\right] \biggr|\\
&\leq \sum_{t=0}^{\infty} \mathbb{E}\left[L d_H(X_t,Y_t) \mid X_0 = x^{(i,+)}, Y_0 = x^{(i,-)}\right] \\
&\leq L\sum_{t=0}^{\infty} \alpha^t \\
&= \frac{L}{1-\alpha}\,.
\end{align*}

Let $g: \mathbb{R} \to \mathbb{R}$ be any 1-Lipschitz function. Let $h_g$ be the solution to the Poisson equation $h_g - Ph_g = g\circ f - \mathbb{E}_\mu (g\circ f)$.
To prove Equation~\eqref{contractive_chain_wasserstein_bound}, it is sufficient (by definition of Wasserstein distance) to show that for any 1-Lipschitz function $g$, $\Delta_i(h_g) \leq \frac{L}{1-\alpha}$. By Lemma~\ref{principal_solution},
\begin{align*}
|\Delta_i(h_g)(x)| &=  \biggr|\sum_{t=0}^{\infty} \mathbb{E}\left[g\circ f(X_t) - g\circ f(Y_t) \mid X_0 = x^{(i,+)}, Y_0 = x^{(i,-)}\right] \biggr|\\
&\leq \sum_{t=0}^{\infty} \mathbb{E}\left[|f(X_t) - f(Y_t)| \mid X_0 = x^{(i,+)}, Y_0 = x^{(i,-)}\right]\\
&\leq \sum_{t=0}^{\infty} \mathbb{E}\left[L d_H(X_t,Y_t) \mid X_0 = x^{(i,+)}, Y_0 = x^{(i,-)}\right] \,.
\end{align*}
The bound from the previous display now gives the result.
\end{proof}

Roughly speaking, according to Theorem~\ref{main_theorem}, if $\frac{1}{n}\sum_{i=1}^{n}\|\mu_i(\cdot|x^{(\sim i)})-\nu_i(\cdot|x^{(\sim i)})\|_{\mathsf{TV}}$ is small and $\Delta_i(h)$ is not too large, then $\mathbb{E}_\mu f \approx \mathbb{E}_\nu f$. The quantity $\Delta_i(h)$ is assured to be small if $f$ is Lipschitz and the chain is contractive, and this gives us a bound on the Wasserstein distance. In our main application we deal with chains which are not contractive everywhere and we use the stronger bound~\eqref{main_bound} to obtain results similar to~\eqref{contractive_chain_bound} and~\eqref{contractive_chain_wasserstein_bound}.

\section{Ising Model and Approximation Results}
\label{s:IsingApproximation}
\subsection{Ising model}
\label{subsec:ising_model_def}
We now consider the Ising model. The \emph{interaction matrix} $J$ is a real-valued symmetric $n \times n$ matrix with zeros on the diagonal. Define the Hamiltonian $\mathcal{H}_J: \{-1,1\}^n \to \mathbb{R}$ by $$\mathcal{H}_J(x) = \frac{1}{2}x^{\intercal}Jx\,.$$ Construct the graph $G_J = ([n],E_J)$ with $(i,j) \in E$ iff $J_{ij} \neq 0$.  An Ising model over graph $G_J$ with interaction matrix $J$ is the probability measure $\pi$ over $\{-1,1\}^n$ such that $\pi(x) \propto \exp{(H_J(x))}$. We call the Ising model ferromagnetic if $J_{ij} \geq 0$ for all $i,j$.

For any simple graph $G = ([n],E)$ there is associated a symmetric $n\times n$ adjacency matrix $\mathcal{A}(G) = (\mathcal{A}_{ij})$, where
\begin{equation*}
\mathcal{A}_{ij} = 
\begin{cases}
1 &\text{if  $(i,j) \in E$} \\
0 &\text{otherwise}\,.
\end{cases}
\end{equation*}
Let $K_n$ be the complete graph over $n$ nodes; we will use $A$ to denote its adjacency matrix.
The \emph{Curie-Weiss model} at inverse temperature $\beta > 0$ is an Ising model with interaction matrix $\frac{\beta}{n}A$. It is known that the Curie-Weiss model undergoes phase transition at $\beta = 1$~\cite{ellis2007entropy}. 
We henceforth denote by $\mu$ the Curie-Weiss model at inverse temperature $\beta$. 

We will compare Ising models on the complete graph to those on a $d$-regular graph $G_d = ([n],E_d)$ (i.e., every node has degree $d$). Let $B$ denote the adjacency matrix of $G_d$.  Given inverse temperature $\beta$, we take $\nu$ to be the Ising model with interaction matrix $\frac{\beta}{d}B$.

\subsection{Expander Graphs}

We recall that $A$ is set to be the adjacency matrix of $K_n$. The all-ones vector $\mathbf{1}=[1,1,\dots,1]^{\intercal}$ is an eigenvector of $A$ with eigenvalue $n-1$. It is also an eigenvector of $B$ with eigenvalue $d$. $B$ has the following spectral decomposition with vectors $v_i$ being mutually orthogonal and orthogonal to $\mathbf{1}$:
\begin{equation}
B = \frac{d}{n}\mathbf{1}\mathbf{1}^{\intercal} + \sum_{i=2}^{n}\lambda_i v_iv_i^{\intercal}\,.
\label{spectrum_1}
\end{equation}
Because of the degeneracy of the eigenspaces of $A$, we can write:
\begin{equation}
A = \frac{n-1}{n}\mathbf{1}\mathbf{1}^{\intercal} + \sum_{i=2}^{n}v_iv_i^{\intercal}\,.
\label{spectrum_2}
\end{equation}

Let $\epsilon \in (0,1)$. We call the graph $G_d$ an $\epsilon$-expander if the eigenvalues $\lambda_2,\dots,\lambda_n$ of its adjacency matrix $B$ satisfy $|\lambda_i| \leq \epsilon d$. Henceforth, we assume that $G_d$ is an $\epsilon$-expander. Then, from~\eqref{spectrum_1} and~\eqref{spectrum_2} we conclude that
\begin{equation}
\|\tfrac{\beta}{n}A-\tfrac{\beta}{d}B\|_2 \leq \beta (\epsilon + \tfrac{1}{n})\,.
\label{norm_bound}
\end{equation}
Expanders have been extensively studied and used in a variety of applications. There are numerous explicit constructions for expander graphs. A famous result by Alon and Boppana \cite{nilli1991second} shows that $\epsilon \geq  2\frac{\sqrt{d-1}}{d}$ for any $d$-regular graph. A family of $d$-regular graphs with increasing number of nodes is called \emph{Ramanujan} if $\epsilon$ approaches $2\frac{\sqrt{d-1}}{d}$ asymptotically. A $d$-regular graph over $n$ nodes is said to be $\delta$-approximately Ramanujan if $\epsilon = 2\frac{\sqrt{d-1} + \delta}{d}$. \cite{friedman2008proof} shows that for every $\delta > 0$, a random $d$-regular graph is $\delta$-approximately Ramanujan with probability tending to $1$ as $n \to \infty$.

Our main result in Subsection~\ref{ss:mainIsingMoments} is a bound on the difference of low-order moments of $\mu$ and $\nu$. 
Before discussing this, we warm up by applying our method to Ising models in the contracting regime.

\subsection{Approximation of Ising Models under Dobrushin-like Condition} 
\label{ss:Dobrushin}

In Theorem~\ref{t:contractIsing} below, we use the fact that Ising models contract when the interactions are weak enough to prove bounds on the Wasserstein distance between functionals of two Ising models.
Given $x,y \in \Omega$, let $\Delta_{x,y}$ denote the column vector with elements $\frac{1}{2}|x_i - y_i| = \ind_{x_i \neq y_i}$. Let $|L|$ be the matrix with entries $(|L|)_{ij} = |L_{ij}|$ equal to the absolute values of entries of $L$. The Ising model with interaction matrix $L$ is then said to satisfy the \emph{Dobrushin-like condition} if $\|(|L|)\|_2 <1$. Essentially the same condition was used in \cite{hayes2006simple} and \cite{chatterjee2005concentration}.
This contrasts with the classical Dobrushin condition, which requires that $\|(|L|)\|_{\infty} < 1$~\cite{dobrushin1970prescribing,weitz2005combinatorial,georgii2011gibbs}. In both the Curie-Weiss model with interaction matrix $\frac \beta n A$ and the Ising model on $d$-regular graph with interaction matrix $\frac \beta d B$, the Dobrushin-like condition as well as the classical Dobrushin condition are satisfied if and only if $\beta <1$.

\begin{remark}

We state these conditions in terms of the Ising interaction matrix, but in general they use the so-called dependence matrix. We briefly describe the connection. Given a measure $\pi$ over $\Omega$, the matrix $D = (d_{ij})$ is a dependence matrix for $\pi$ if for all $x,y \in \Omega$,
$
\|\pi_i(\cdot| x^{(\sim i)}) - \pi_i(\cdot| y^{(\sim i)})\|_{\mathsf{TV}} \leq \sum_{j=1}^{n} d_{ij}\ind_{x^i \neq y^i}\,.
$
The measure $\pi$ satisfies the Dobrushin condition with dependence matrix $D$ if $\|D\|_\infty < 1$. If $\pi$ is an Ising model with interaction matrix $J$, then $\pi_i(x_i =1|x^{(\sim i)}) = \frac12(1+\tanh{J_i^{\intercal}x})$ (here $J_i^{\intercal}$ is the $i$th row of $J$). Therefore,
$
\|\pi_i(\cdot| x^{(\sim i)}) - \pi_i(\cdot| y^{(\sim i)})\|_{\mathsf{TV}} = \frac{1}{2} |\tanh{J_i^{\intercal}x} - \tanh{J_i^{\intercal}y}| 
\leq  \sum_{j=1}^{n} |J_{ij}| \ind_{x^i \neq y^i}
$
and we can consider $|J|$ as the dependence matrix. \end{remark}

For $a\in (\mathbb{R}^+)^n$, let $f: \Omega \to \mathbb{R}$ be any function such that $\forall \ x,y \in \Omega$, 
$$
|f(x) - f(y)| \leq \sum_{i=1}^n a_i \ind_{x_i \neq y_i} = a^{\intercal}\Delta_{x,y}\,.
$$
We call such a function $f$ an $a$-Lipschitz function.

\begin{theorem}\label{t:contractIsing}
	Let $a \in (\mathbb{R}^{+})^n$ and let $f:\Omega \to \mathbb{R}$ be an $a$-Lipschitz function. If an interaction matrix $L$ (with corresponding Ising measure $\pi_L$) satisfies the Dobrushin-like condition, then for any other interaction matrix $M$ (with corresponding Ising measure $\pi_M$), 
	$$|\mathbb{E}_{\pi_L} f - \mathbb{E}_{\pi_M} f| 
	\leq \frac{\|a\|_2 \sqrt{n}}{2(1-\|(|L|)\|_2)} \|L-M\|_2\,.$$
\end{theorem}
The proof, given in Section~\ref{s:contractProof}, uses ideas from Section~4.2 in~\cite{chatterjee2005concentration}, which proves results on concentration of Lipschitz functions of weakly dependent random variables.

A simple consequence of this theorem is that when $\|(|L|)\|_2 < 1$, the Ising model is stable in the Wasserstein distance sense under small changes in inverse temperature. 

\begin{corollary}
Let $M = (1+ \epsilon)L$. Then, for any $a$-Lipschitz function,
$$ |\mathbb{E}_{\pi_L}f - \mathbb{E}_{\pi_M}f| \leq \epsilon\|a\|_2 \sqrt{n}\frac{\|L\|_2}{2(1-\|(|L|)\|_2)} \,.$$ 
\end{corollary}

If $f$ is $\tfrac{1}{n}$-Lipschitz in each coordinate 
then $\|a\|_2 = \frac{1}{\sqrt{n}}$ (typical statistics like magnetization fall into this category). We conclude that for such functions
$$ |\mathbb{E}_{\pi_L}f - \mathbb{E}_{\pi_M}f| \leq \frac{\epsilon\|L\|_2}{2(1-\|(|L|)\|_2)} \,.$$

 
 \subsection{Main Result on Approximation of Ising Model Moments}
 \label{ss:mainIsingMoments}
 
 Let $\rho_{ij} = \mathbb{E}_{\mu} x^ix^j$ and $\tilde{\rho}_{ij} = \mathbb{E}_{\nu} x^ix^j$ denote the pairwise correlations in the two Ising models $\mu$ and $\nu$. It follows from Griffith's inequality~\cite{griffiths1967correlations} for ferromagnetic Ising models that for any $i$ and $j$,
 $$ 0\leq \rho_{ij} \leq 1
 \quad \text{and}\quad 0 \leq \tilde{\rho_{ij}} \leq 1\,.$$
 If two Ising models have the same pairwise correlations for every $i,j \in [n]$, then they are identical. For an Ising model $\eta$ with interaction matrix $J$, it is also not hard to show that if there are no paths between nodes $i$ and $j$ in the graph $G_J$, then $x^i$ and $x^j$ are independent and  $\mathbb{E}_{\eta} [x^ix^j]  = 0$. We refer to \cite{wainwright2008graphical} for proofs of these statements. We conclude that ${{{n}\choose{2}}}\inv \sum_{ij} |\rho_{ij} - \tilde{\rho}_{ij}|$ defines a metric on the space of Ising models over $n$ nodes.
 
 For positive even integers $k$, we denote the $k$th order moments for $i_1,..,i_k \in [n]$ by
 $$\rho^{(k)}[i_1,\dots,i_k] = \mathbb{E}_\mu\Big(\prod_{s=1}^k x^{i_s}\Big)$$
 and similarly for $\tilde \rho^{(k)}[i_1,\dots,i_k]$, but with $\mu$ replaced by $\nu$.
For a set $R = \{i_1,\dots,i_k\}$, we write $\rho^{(k)}[R]$ in place of $\rho^{(k)}[i_1,\dots,i_k]$.
 (We consider only even $k$, since odd moments are zero for Ising models with no external field.)
 
 
 Using Theorem~\ref{main_theorem}, we show the following approximation result on nearness of moments of the Curie-Weiss model and those of the Ising model on a sequence of regular expanders.
 
 \begin{theorem}
 	\label{main_application}
 	Let $A$ be the adjacency matrix of the complete graph and let $B$ be the adjacency matrix of a $d$-regular $\epsilon$-expander, both on $n$ nodes. Let the inverse temperature $\beta > 1$ be fixed, and consider the Ising models with interaction matrices $\frac\beta n A$ and $\frac \beta d B$, with moments $\rho$ and $\tilde \rho$ as described above. There exist positive constants $\epsilon_0(\beta)$ and $C(\beta)$ depending only on $\beta$ such that if $\epsilon < \epsilon_0(\beta)$, then for any even positive integer $k < n$
 	$$\frac{1}{{{n}\choose{k}}} \sum_{\substack{R \in [n] \\ |R| = k}}\big|\rho^{(k)}[R] - \tilde{\rho}^{(k)}[R]\big| \leq kC(\beta)\Big(\epsilon + \frac{1}{n}\Big)\,.$$
 	In particular,
 	$$\frac{1}{{{n}\choose{2}}} \sum_{i j}|\rho_{ij} - \tilde{\rho}_{ij}| < 2C(\beta) \Big(\epsilon + \frac{1}{n}\Big)\,.$$
 \end{theorem}

 For approximately Ramanujan graphs, $\epsilon = \Theta(\frac{1}{\sqrt{d}})$. By choosing a random $d$-regular graph, which is approximately Ramanujan with high probability, we can obtain arbitrarily accurate approximation of moments by choosing $d$ sufficiently large. If we care only about moments up to some fixed order $\bar k$, our result says that one can take any $d = \Omega(\bar k ^2)$ in order to obtain the desired approximation, completely independent of the size of the graph.

 The structure of the approximating graph $G_d$ is important. To see this, let the graph $G_d$ be the disjoint union of $\frac{n}{d}$ cliques each with $d$ nodes, a poor spectral sparsifier of the complete graph $K_n$. Consider the Ising model with interaction matrix $\frac{\beta}{d}\mathcal{A}(G_d)$. This graph is not an expander since it is not connected. If $i$ and $j$ are in different cliques, there is no path between $i$ and $j$ in $G_d$. Therefore, $\tilde{\rho}_{ij} = 0$.  We conclude that only $O(\frac{d}{n})$ fraction of the pairs $(i,j)$ have correlation $\tilde{\rho}_{ij} > 0$. Since $\beta > 1$, it follows by standard analysis for the Curie-Weiss model that $\rho_{ij} > c_1(\beta) >0$ (see \cite{ellis2007entropy}).
 Therefore,
 \begin{align*}
 \frac{1}{{{n}\choose{2}}} \sum_{i j} |\rho_{ij} - \tilde{\rho}_{ij}| &\geq \frac{1}{{{n}\choose{2}}}\sum_{i j} (\rho_{ij} - \tilde{\rho}_{ij})\\
 &\geq c_1(\beta) - O\Big(\frac{d}{n}\Big)\,.
 \end{align*}
 Here we have used the fact that $\tilde{\rho}_{ij} \leq 1$. It follows that  if $\beta > 1$ and $d = o(n)$, then the left-hand side cannot be made arbitrarily small.
 
 The case $0 \leq\beta < 1$ is trivial in the sense that the average correlation is very small in both models and hence automatically well-matched.
 \begin{proposition}
 	 	Consider the same setup as Theorem~\ref{main_application}, but with $0\leq \beta<1$. Then both $\sum_{i\neq j}{\rho_{ij}}=O(n)$ and $\sum_{i\neq j}{\tilde{\rho}_{ij}}=O(n)$, and hence
 	 	$${{n}\choose{2}}\inv \sum_{j}\sum_{i< j}|\rho_{ij} - \tilde{\rho}_{ij}|\leq {{n}\choose{2}}\inv \sum_{j}\sum_{i< j}(\rho_{ij} + \tilde{\rho}_{ij} ) =O\left(\tfrac1n\right)\,.$$
 \end{proposition}

\begin{proof}
 To start, note that
 \begin{align}
 \sum_{i\neq j}{\rho_{ij}} &= \mathbb{E}_{\mu} \Big(\sum_{i=1}^n x^i\Big)^2 - n  = \mathrm{var}_{\mu}\Big(\sum_{i=1}^n x^i\Big) -n \qquad \text{and}
 \label{eq:correlation_sum_1}
\\
 \sum_{i\neq j}{\tilde{\rho}_{ij}}&=\mathbb{E}_{\nu} \Big(\sum_{i=1}^n x^i\Big)^2 - n  = \mathrm{var}_{\nu}\Big(\sum_{i=1}^n x^i\Big) -n \,.
 \label{eq:correlation_sum_2}
 \end{align}
 Thus, it suffices to show that the variances on the right-hand sides are $O(n)$. 
 In the equations above, $\mathrm{var}_{\eta}(f)$ refers to variance of $f$ with respect to measure $\eta$. We bound the variance for the measure $\mu$ and identical arguments can be used to bound the variance with respect to $\nu$. 
 
 Whenever $\beta < 1$, from the proof of Theorem 15.1 in \cite{levin2009markov}, we conclude that Glauber dynamics for both these models is $1-\frac{1- \beta}{n}$ contracting. Let $(\lambda_i)_{i=1}^{|\Omega|}$ be the eigenvalues of $P$. We let $|\lambda| := \sup \{1-|\lambda_i|: \lambda_i \neq 1, 1\leq i\leq |\Omega|\}$. From Theorem 13.1 in \cite{levin2009markov}, it follows that the spectral gap, $1-|\lambda|\geq \frac{1-\beta}{n}$. For any function $f : \Omega \to \mathbb{R}$, the Poin\'care inequality for $P$ bounds the variance under the stationary measure as $\mathrm{var}_\mu(f)\leq \frac12(1-|\lambda|)\inv  \mathcal{E}(f,f)$,
where the Dirichlet form
 $\mathcal{E}(f,f):= \sum_{x,y\in \Omega}(f(x)-f(y))^2P(x,y)\mu(x)$
 (see Subsection 13.3 in \cite{levin2009markov}). The Poin\'care inequality then becomes
 \begin{align}\mathrm{var}_{\mu}(f)&\leq \frac{1}{2(1-|\lambda|)} \sum_{x,y\in \Omega}(f(x)-f(y))^2P(x,y)\mu(x) \nonumber\\
 &\leq \frac{n}{2(1-\beta)}\sum_{x,y\in \Omega}(f(x)-f(y))^2P(x,y)\mu(x)\,.
 \label{eq:poincare_inequality}
 \end{align}

Since $P$ is the Glauber dynamics, $P(x,y) >0$ only when $x$ and $y$ differ in at most one coordinate. When we take $f(x) = \sum_i x^i$, then $|f(x) -f(y)| \leq 2$ whenever $P(x,y)>0$. Plugging this into Equation~\eqref{eq:poincare_inequality} yields
$$\mathrm{var}_{\mu}\Big(\sum_i x^i\Big)\leq \frac{2n}{1-\beta} = O(n)\,,$$
and similarly
$\mathrm{var}_{\nu}\big(\sum_i x^i\big) = O(n)$.
%
 \end{proof}

\section{Monotone Coupling and Proof of Theorem~\ref{t:contractIsing}}
\label{s:contractProof}

\subsection{Glauber Dynamics for Ising models and Monotone Coupling}
\label{subsec:monotone_coupling} 
We specialize our previous discussion of the Glauber dynamics in Subsection~\ref{ss:glauber} to an Ising model with interaction matrix $J$. Let $J^{\intercal}_i$ denote the $i$th row of $J$. Given the current state $x \in\Omega= \{-1,1\}^n$, the Glauber Dynamics produces the new state $x^{\prime}$ as follows:

Choose $I \in [n]$ uniformly at random. Construct the next state $x^{\prime}$ as $(x^{\prime})^i= x^i$ for $i\neq I$ and set independently
\begin{equation*}
(x^{\prime})^I = \begin{cases}
1 &\text{ with probability $\tfrac{1}2+ \tfrac12\tanh{J^{\intercal}_Ix}$ }
\\
-1 &\text{ with probability $\tfrac{1}2- \tfrac12\tanh{J^{\intercal}_Ix}$}\,.
\end{cases}
\end{equation*}
We refer to \cite{levin2009markov} for an introduction to mixing of Glauber dynamics for the Ising model. This Markov chain has been studied extensively and it can be shown that it mixes in $O(n\log{n})$ time (and is contracting for the `monotone coupling' described below) for high temperature under the Dobrushin-Shlosman condition~\cite{aizenman1987rapid} and under Dobrushin-like condition~\cite{hayes2006simple}.

We now describe the monotone coupling used in the proof of Theorem~\ref{t:contractIsing}. Let $X_t$ and $Y_t$ be Glauber dynamics chains for the Ising model $\pi_J$ with interaction matrix $J$. Let $P^J$ denote the corresponding kernel. 
 For both chains $X_t$ and $Y_t$, we choose the same random index $I$ and generate an independent random variable $u_t \sim \mathrm{unif}([0,1])$. Set $X_{t+1}^I$ (resp. $Y_{t+1}^I$) to $1$ iff $u_t \leq (\pi_J)_I\big(1\big| X_t^{(\sim I)}\big)$ (resp. $u_t \leq (\pi_J)_I\big(1\big| Y_t^{(\sim I)}\big)$ ). In the case when the entries of $J$ are all positive (i.e, ferromagnetic interactions), one can check that for the coupling above, if $X_0 \geq Y_0$ then $X_t \geq Y_t$ a.s. We note that since $J$ need not be ferromagnetic in the case considered in Theorem~\ref{t:contractIsing}, we cannot ensure that $X_t \geq Y_t$ a.s. if $X_0 \geq Y_0$. (Here $\geq$ is the entrywise partial order.)

\subsection{Auxiliary Lemma} 

 Before proceeding with the proof of Theorem~\ref{t:contractIsing}, we prove the following lemma that relates the quantity we wish to bound to the spectral norm of Ising interaction matrices.
 
\begin{lemma}\label{spectral_bound}
Let $f_1(x),\dots, f_n(x)$ be any real valued functions over $\Omega$ and define the vector $v_f(x) = [f_1,\dots,f_n(x)]^{\intercal}$. Let $\pi_L$ and $\pi_M$ denote Ising models with interaction matrices $L$ and $M$ respectively. Then,
$$\frac{1}{n}\sum_{i=1}^{n}|f_i(x)|\cdot\|(\pi_L)_i(\cdot|x^{(\sim i)})- (\pi_M)_i(\cdot|x^{(\sim i)})\|_{\mathsf{TV}} \leq \|L-M\|_2 \frac{\|v_f\|_2}{2\sqrt{n}}\,.$$ 
In particular, when $L = \frac{\beta}{n}A$ (i.e, $\pi_L(\cdot) = \mu(\cdot)$, the Curie-Weiss model at inverse temperature $\beta$) and $M = \frac{\beta}{d}B$ (i.e, $\pi_M(\cdot) = \nu(\cdot)$, where $\nu(\dot)$ is the Ising model defined in Section~\ref{subsec:ising_model_def}) then
$$\frac{1}{n}\sum_{i=1}^{n}|f_i(x)|\cdot\|\mu_i(\cdot|x^{(\sim i)})- \nu_i(\cdot|x^{(\sim i)})\|_{\mathsf{TV}} \leq \|\tfrac{\beta}{n}A - \tfrac{\beta}{d}B\|_2 \frac{\|v_f\|_2}{2\sqrt{n}}\,.$$ 
\end{lemma}
\begin{proof}
The proof follows from the $1$-Lipschitz property of the $\tanh(\cdot)$ function. Let $L^{\intercal}_i$ denote the $i$th row of $L$. We recall that $(\pi_L)_i(1|x^{(\sim i)}) = \frac{1}{2}(1+\tanh{L_i^{\intercal}x})$. There exist $c_i(x) \in \{-1,1\}$ such that the following holds, where we use the notation $v_{cf}^{\intercal}(x) = [c_1(x)f_1(x),\dots.,c_n(x)f_n(x)]$:
\begin{align*}
&\frac{1}{n}\sum_{i=1}^{n}|f_i(x)|\cdot\|(\pi_L)_i(\cdot|x^{(\sim i)})- (\pi_M)_i(\cdot|x^{(\sim i)})\|_{\mathsf{TV}}
\\ &\qquad= \frac{1}{2n}\sum_{i=1}^{n}|f_i(x)|\cdot \big|\tanh{(L_i^{\intercal}x)} - \tanh{(M_i^{\intercal}x)}\big| \\
&\qquad\leq \frac{1}{2n} \sum_{i=1}^{n}|f_i(x)|\cdot\big|\big(L_i^{\intercal} - M_i^{\intercal}\big)x\big| \\
&\qquad=  \frac{1}{2n} \sum_{i=1}^{n}c_i(x)f_i(x)\big(L_i^{\intercal} - M_i^{\intercal}\big)x \\
&\qquad= \frac{1}{2n} v_{cf}^{\intercal}(x)(L-M)x \\
&\qquad\leq \frac{1}{2n}\|x\|_2 \|L - M\|_2 \|v_{cf}\|_2\\
&\qquad= \|L-M\|_2 \frac{\|v_f\|_2}{2\sqrt{n}}\,.& \qedhere
\end{align*}
\end{proof}


\subsection{Proof of Theorem~\ref{t:contractIsing}}
Let $h$ be the solution of the Poisson equation
$$h - P^L h = f - \mathbb{E}_{\pi_L}f\,.$$
In order to apply Theorem~\ref{main_theorem}, we bound the quantity
\begin{align}
|\Delta_i(h)(x)| &= \biggr\rvert\sum_{t=0}^{\infty}\mathbb{E}\Big[f(X_t) - f(Y_t)\Big| X_0 = x^{(i,+)}, Y_0 = x^{(i,-)}\Big]\biggr\rvert \notag
\\
&\leq \sum_{t=0}^{\infty}\mathbb{E}\bigg[\sum_{i=1}^n a_i \ind_{X^i_t \neq Y^i_t}\biggr\rvert X_0 = x^{(i,+)}, Y_0 = x^{(i,-)}\bigg]\notag
\\ \label{e:deltah}
&= \sum_{t=0}^{\infty}\mathbb{E}\Big[ a^{\intercal}\Delta_{X_t, Y_t}\Big| X_0 = x^{(i,+)}, Y_0 = x^{(i,-)}\Big]\,.
\end{align}

The equation above holds for all couplings between $X_t$ and $Y_t$. We choose the monotone coupling as described in Section~\ref{subsec:monotone_coupling}. We recall that $(\pi_L)_i(1|x^{\sim i}) = \frac12(1+\tanh{L_i^{\intercal}x})$. From the definition of Glauber dynamics and monotone coupling it follows that
$$\mathbb{E}\big[\ind_{X^i_t \neq Y_t^i}| X_{t-1}=x, Y_{t-1} = y\big] = (1-\tfrac{1}{n})\ind_{x_i \neq y_i} + \frac{1}{2n}|\tanh{L_i^{\intercal}x} - \tanh{L_i^{\intercal}y}|\,.$$
If $c$ is an $n$-dimensional column vector with positive entries, then
\begin{align*}
\mathbb{E}\big[  c^{\intercal}\Delta_{X_{t+1},Y_{t+1}}\big\rvert X_t = x, Y_t = y\big]&= \mathbb{E}\Big[\sum_{i=1}^n c_i \ind_{X^i_{t+1} \neq Y^i_{t+1}}\Big| X_t = x, Y_t = y\Big] \\
&\leq \sum_{i=1}^n (1-\tfrac{1}{n})a^{\intercal}\Delta_{x,y} \\&\qquad+ \tfrac{1}{2n}\sum_{i=1}^{n}a_i |\tanh{L_i^{\intercal }x} - \tanh{L_i^{\intercal }y}| \\
&\leq c^{\intercal}\left[(1-\tfrac{1}{n})I+ \tfrac{1}{n}|L| \right] \Delta_{x,y}\\
&= c^{\intercal}G \Delta_{x,y} \,,
\end{align*}
where  $G := (1-\tfrac{1}{n})I+ \tfrac{1}{n}|L| $. Clearly, $\|G\|_2 < 1$ and hence $\sum_{t=0}^{\infty} G^t = (I-G)^{-1}$.
Using the tower property of conditional expectation to apply the above inequality recursively, we conclude that
$$\mathbb{E}\Big[  c^{\intercal}\Delta_{X_{t+1},Y_{t+1}}\Big| X_0 = x^{(i,+)}, Y_0 = x^{(i,-)}\Big] \leq  c^{\intercal}G^{t+1}\Delta_{x^{(i,+)},x^{(i,-)}} \,.$$
Plugging the equation above into~\eqref{e:deltah} gives
\begin{align*}
|\Delta_i(h)(x)| \leq a^{\intercal}\bigg[\sum_{t=0}^{\infty} G^t \bigg]\Delta_{x^{(i,+)},x^{(i,-)}} 
&=  a^{\intercal} (I-G)^{-1}\Delta_{x^{(i,+)},x^{(i,-)}} \\
&= \big[a^{\intercal}(I-G)^{-1}\big]^i \,.
\end{align*}
Recall that $\theta := \|(|L|)\|_2 < 1$, which implies that
\begin{align*}
\sqrt{\sum_{i=1}^n \Delta_i(h)^2} \leq \sqrt{\sum_{i=1}^n \big(\left[a^{\intercal}(I-G)^{-1}\right]^i\big)^2} 
&= \|a^{\intercal}(I-G)^{-1}\|_2 \\
&\leq \|a\|_2 \|(I-G)^{-1}\|_2 \\
&\leq \|a\|_2 \frac{1}{1-\|G\|_2}\\
&= \|a\|_2 \frac{n}{1-\theta} \,.
\end{align*}
We invoke Lemma~\ref{spectral_bound} and Theorem~\ref{main_theorem} to complete the proof: 
\begin{align*}
|\mathbb{E}_{\pi_L} f - \mathbb{E}_{\pi_M} f| &\leq \mathbb{E}_{\pi_M}\sqrt{\sum_{i=1}^n \Delta_i(h)^2} \frac{\|L-M\|_2}{2\sqrt{n}} \\
&\leq \frac{\|a\|_2 \sqrt{n}}{2(1-\theta)} \|L-M\|_2\,.
\end{align*}

\section{Ideas in Proof of Theorem~\ref{main_application}}
\label{s:ProofIdeas}

\subsection{Overview}
In this section we overview the main ideas behind the proof of Theorem~\ref{main_application}, which bounds the average difference in $k$th order moments in the Curie-Weiss model $\mu$ and the $d$-regular Ising model $\nu$. 

Let $k$ be any even positive integer such that $k < n$.
For every $R \subset [n]$ such that $|R| = k$, let $C_{R} \in \{-1,1\}$ and define the function $f_C :\Omega \to \mathbb{R}$
\begin{equation}\label{e:f}
	f_C(x) = \frac{1}{2k{{n}\choose{k}}}\sum_{\substack{R \subset [n]\\ |R| = k}} C_{R}\prod_{i\in R}x^i\,.
\end{equation}
We suppress the subscript $C$ in $f_C$. 
Clearly, $f(x) = f(-x)$, i.e., $f$ is symmetric. Moreover, a calculation shows that $f$ is $\frac{1}{n}$-Lipschitz with respect to the Hamming metric. That is, for arbitrary $x,y \in \Omega$,
$
|f(x)-f(y)| \leq \frac{1}{n}\sum_{i=1}^{n}\ind(x^i \neq y^i)
$,
which implies that $|f(x) - f(y)| \leq 1$ for any $x,y \in \Omega$.
In Section~\ref{s:IsingProof}
we will bound the quantity $|\E_\mu f - \E_\nu f|$ uniformly for any choice of $\{C_R\}$, which in turn relates the moments $\rho$ and $\tilde \rho$ (defined in Subsection~\ref{ss:mainIsingMoments}) since
$$
\sup_{C_R}|\E_\mu f - \E_\nu f|  = \frac{1}{2k{{n}\choose{k}}}\sum_{\substack{R \subset [n]\\ |R| = k}} \big|\rho^{(k)}[R] - \tilde{\rho}^{(k)}[R]\big|\,.
$$

Let $P$ be the kernel of the Glauber dynamics of the Curie-Weiss model at inverse temperature $\beta > 1$. By Theorem~\ref{main_theorem}, bounding $|\E_\mu f - \E_\nu f|$ reduces to bounding $\mathbb{E}_{\nu}|\Delta_i(h)|$ for the specific function $h$ obtained by solving the Poisson equation $(I-P)h = f - \mathbb{E}_{\mu}f$.

 By Lemma~\ref{principal_solution}, we can write $h$ in terms of the expectation of a sum over time-steps for Glauber chains $X_t$ and $Y_t$ to obtain
\begin{align}\label{e:h}
{h}(x) - {h}(y)  &= \mathbb{E}\bigg[\sum_{t=0}^{\infty}f({X}_t) - f({Y}_t)\biggr\rvert {X}_0=x , {Y}_0 =y\bigg] 
\end{align}
from which we get
$$
|\Delta_i(h)(x_0)| = \Big|\sum_{t=0}^{\infty}\mathbb{E}\Big[\big(f({X}_t) - f({Y}_t)\big)\Big|{X}_0=x_0^{(i,+)} , {Y}_0 =x_0^{(i,-)}\Big]\Big| \,.
$$
 By selecting $x_0\sim \nu$, 
this yields a method for bounding $\mathbb{E}_{\nu}|\Delta_i(h)|$ via coupling $X_t$ and $Y_t$.

 We now briefly overview the steps involved in bounding $\mathbb{E}_{\nu}|\Delta_i(h)|$.
Let $m^{*}$ be the unique positive solution to $ s =\tanh{\beta s}$. 
\begin{enumerate}
\item[Step 1:]  
For a good enough expander, $m(x):= \frac1n\sum_i x^i$ concentrates exponentially near $m^{*}$ and $-m^{*}$ under measure $\nu$. We show this in Lemma~\ref{expander_magnetisation_concentration}. The subsequent analysis is separated into two cases depending on whether or not $m(x_0)$ is close to $m^*$. 
\item[Step 2:] Theorem~\ref{main_theorem} requires specifying a Markov kernel; because the Glauber dynamics on the Curie-Weiss model mixes slowly when $\beta >1$, we instead use the \emph{restricted} (a.k.a. censored) Glauber dynamics, which restricts the Glauber dynamics to states with majority of $+1$ coordinates and mixes quickly. We justify this change with Lemma~\ref{symmetric_solution}.
\item[Step 3:] Whenever $m(x)$ is not close to $m^{*}$, we show in Lemma~\ref{naive_bound} that $|\Delta_i(h)(x)|$ is at most polynomially large in $n$. This is achieved via coupling $X_t$ and $Y_t$ in \eqref{e:h} and makes use of fast mixing of the chain.
\item[Step 4:] Whenever $m(x)$ is near enough to $m^{*}$, the restricted Glauber dynamics (and Glauber dynamics) for the Curie-Weiss model is contracting for a certain coupling. Using methods similar to the ones used in the proof of Theorem~\ref{t:contractIsing} in the contracting case, we conclude that $|\Delta_i(h)|(x)$ must be small if $m(x)$ is close to $m^*$. We show this in Section~\ref{s:coupling} via Lemmas~\ref{super_martingale_contraction},~\ref{contraction_bound_lemma} and Theorem~\ref{contraction_bound}. 
\item[Step 5:] Section~\ref{s:IsingProof} combines these statements to bound $\mathbb{E}_{\nu}|\Delta_i(h)|$ and prove Theorem~\ref{main_application}.
\end{enumerate}

\subsection{Concentration of Magnetization}
Recall that $m^{*}$ is the largest solution to the equation $\tanh{\beta s} = s$. If $\beta \leq 1$, then $m^{*}=0$ and if $\beta >1$, then $m^{*} > 0$. Recall the magnetization 
$m(x) := \frac1{n}{\sum_{i=1}^{n}x^i}$. Whenever it is clear from context, we denote $m(x)$ by $m$.
\begin{lemma}\label{expander_magnetisation_concentration}
For every $\delta \in (0,1)$, there exists $c(\delta) >0$ and $\epsilon_0(\delta) > 0$ such that for all $\epsilon$-expanders $G_d$ with $\epsilon < \epsilon_0$, 
$$
\nu\big(\{|m-m^{*}| > \delta\}\cap \{|m+m^{*}| > \delta\}\big) \leq C_1(\beta)e^{-c(\delta)n}\,.
$$
\end{lemma}
The proof is essentially the same as the proof of concentration of magnetization in the Curie-Weiss model, but with a few variations. We defer the proof to the appendix.

\subsection{Restricted Glauber Dynamics}

Glauber dynamics for the Curie-Weiss model is well-understood and it can be shown to mix in $O(n\log{n})$ time when $\beta < 1$, $O(n^{\frac{3}{2}})$ time when $\beta = 1$, and takes exponentially long to mix when $\beta >1$ (see \cite{levin2010glauber} and references therein). The reason for exponentially slow mixing is that it takes exponential time for the chain to move from the positive phase to the negative phase and vice-versa. The Restricted Glauber Dynamics, described next, removes this barrier.  

Define $\Omega^{+} = \{x \in \Omega : \sum_{i}x^i \geq 0\}$. \cite{levin2010glauber} and \cite{ding2009censored} considered a censored/ restricted version of Glauber dynamics for the Curie-Weiss model where the chain is restricted to the positive phase $\Omega^+$. 
Let $\hat{X}_t$ be an instance of restricted Glauber dynamics and let $X'$ be obtained from $\hat{X}_t$ via one step of normal Glauber dynamics. If $X' \in \Omega^{+}$, then 
the restricted Glauber dynamics updates to $\hat X_{t+1} = X'$. Otherwise $X^{\prime} \notin \Omega^+$ and we flip all the spins, setting $\hat{X}_{t+1} = -X'$. 

The restricted Glauber dynamics $\hat{X}_t$ with initial state $\hat{X}_0 \in \Omega^{+}$ can be obtained from the normal Glauber dynamics also in a slightly different way.
Let $X_t$ be a Glauber dynamics chain with $X_0 = \hat{X}_0\in \Omega^{+}$, and let
\begin{equation*}
\hat{X}_t = \begin{cases}
X_t & \text{if $X_t \in \Omega^{+}$} \\ -X_t & \text{if $X_t \notin \Omega^{+}$} \,.
\end{cases}
\end{equation*}
Whenever we refer to restricted Glauber dynamics, we assume that it is generated as a function of the regular Glauber dynamics in this way.

 If $\mu$ is the stationary measure of the original Glauber dynamics, then the unique stationary measure for the restricted chain is $\mu^{+}$ over $\Omega^{+}$, given by
\begin{equation}
\mu^{+}(x) = \begin{cases}
2\mu(x) & \text{if $m(x) > 0$} \\
\mu(x) & \text{if $m(x) = 0$}\,.
\end{cases}
\end{equation}
Similarly, we define $\nu^{+}$ over $\Omega^{+}$ by
\begin{equation}
\nu^{+}(x) = \begin{cases}
2\nu(x)  &\text{if $m(x) > 0$} \\
\nu(x) &\text{if $m(x) = 0$}\,.
\end{cases}
\end{equation}
It follows by symmetry that if $f : \Omega \to \mathbb{R}$ is any function such that $f(x) = f(-x)$, then
$$\mathbb{E}_{\mu} f= \mathbb{E}_{\mu^{+}}f
\quad \text{and}\quad \mathbb{E}_{\nu} f= \mathbb{E}_{\nu^{+}}f\,.$$

It was shown in \cite{levin2010glauber} that restricted Glauber dynamics for the Curie-Weiss model mixes in $O(n\log{n})$ time for all $\beta > 1$.

\begin{theorem}[Theorem~5.3 in \cite{levin2010glauber}] \label{t:Levin}
Let $\beta > 1$. There is a constant $c(\beta) > 0$ so that $t_{mix}(n) \leq c(\beta)n\log{n}$ for the Glauber dynamics restricted to $\Omega^{+}$.
\end{theorem}

\begin{remark}
\label{rem:coupling_existence}
It follows from the proof of the theorem above that there exists a coupling of the restricted Glauber dynamics such that the chains starting at any two distinct initial states will collide in expected time $c(\beta)n\log{n}$. More concretely, let $\hat{X}_t$ and $\hat{Y}_t$ be two instances of restricted Glauber dynamics such that $\hat{X}_0 = x \in \Omega^{+}$ and $\hat{Y}_0 = y \in \Omega^{+}$. Let $\tau_0 =\inf\{t: \hat{X}_t= \hat{Y}_t\}$. There exists a coupling between the chains such that
$$\sup_{x,y \in \Omega^{+}}\mathbb{E}\big[\tau_0\big| \hat{X}_0=x, \hat{Y}_0 = y\big] \leq c(\beta)n\log{n}\,.$$  
and $\hat{X}_t = \hat{Y}_t$ a.s. $\forall \ t \geq \tau_0$.
\end{remark}

\subsection{Solution to Poisson Equation for Restricted Dynamics}
 The next lemma follows easily from the definitions and we omit its proof.

\begin{lemma}\label{symmetric_solution}
Let $f: \Omega \to \mathbb{R}$ be a symmetric function, i.e., for every $x \in \Omega$, $f(x) = f(-x)$. Let $P$ be the kernel of the Glauber dynamics for the Curie-Weiss model at inverse temperature $\beta$ and let $\hat{P}$ be the kernel for the corresponding restricted Glauber dynamics over $\Omega^{+}$ with stationary measure $\mu^{+}$. Then, the Poisson equations
\begin{enumerate}
\item
$h(x) - (Ph)(x) = f(x)-\mathbb{E}_{\mu}f$
\item
$\hat{h}(x) - (\hat{P}\hat{h})(x) = f(x) - \mathbb{E}_{\mu^{+}}f$
\end{enumerate}
have principal solutions $h$ and $\hat{h}$ such that $h(x) = \hat{h}(x)$ for every $x \in \Omega^{+}$ and $h(x) = \hat{h}(-x)$ for every $x \in \Omega\setminus \Omega^{+}$. In particular, $h$ is symmetric.
\end{lemma}
By Lemma~\ref{symmetric_solution}, it is sufficient to solve the Poisson equation, and to bound $\mathbb{E}_\nu |\Delta_i (h)|$, for the restricted Glauber dynamics. Based on Lemmas~\ref{principal_solution} and~\ref{symmetric_solution} we have the following naive bound on the solution of the Poisson equation. 
\begin{lemma}
\label{naive_bound}
Let $f: \Omega \to \mathbb{R}$ be a symmetric function such that for any $x,y \in \Omega$, it holds that: $|f(x)-f(y)| \leq K$. Let $h$ be the solution to the Poisson equation $h - Ph = f - \mathbb{E}_{\mu}f$. Then, for any $x,y\in \Omega$,
$|h(x) - h(y)| \leq KC(\beta)n\log{n}$.
\end{lemma}
\begin{proof}
By Lemma~\ref{symmetric_solution}, $h$ is symmetric and we can without loss of generality assume that $x\in \Omega^+$. Now, we may work with $\hat h$ instead, since
\begin{equation}
h(x) - h(y) = \begin{cases}
\hat{h}(x) - \hat{h}(y) &\text{ if $y \in \Omega^{+}$} \\
\hat{h}(x) - \hat{h}(-y) &\text{ if $y \in \Omega\setminus \Omega^{+}$}\,.
\end{cases}
\end{equation}
Let $x,y \in \Omega^{+}$ and start two restricted Glauber dynamics Markov Chains for Curie-Weiss model $\hat{X}_t$ and $\hat{Y}_t$ with initial states $\hat{X}_0 = x$ and $\hat{Y}_0 = y$.  Recall the definition $\tau_{0} = \inf\{t: \hat{X}_t = \hat{Y}_t\}$ from Remark~\ref{rem:coupling_existence}. We couple $\hat{X}_t$ and $\hat{Y}_t$ according to Remark~\ref{rem:coupling_existence} and use the bound for coupling time, $\mathbb{E}\left[\tau_0|\hat{X}_0 =x, \hat{Y}_0 = y\right] \leq C(\beta)n\log{n}$. By Lemma~\ref{principal_solution}, we can write $\hat h$ in terms of the expectation of a sum to obtain
\begin{align*}
\hat{h}(x) - \hat{h}(y)  &= \mathbb{E}\left[\sum_{t=0}^{\infty}f(\hat{X}_t) - f(\hat{Y}_t)\biggr\rvert \hat{X}_0=x ,\hat{Y}_0 =y\right] \\
&\leq K\cdot \mathbb{E}\left[\sum_{t=0}^{\infty}\ind(\hat{X}_t\neq \hat{Y}_t)\biggr\rvert \hat{X}_0=x ,\hat{Y}_0 =y\right] \\
&= K\cdot \mathbb{E}\Big[\tau_{0}\Big| \hat{X}_0 = x, \hat{Y}_0 = y\Big] \\
&\leq KC(\beta)n\log{n}\,,
\end{align*}
completing the proof.
\end{proof}

The lemma above gives a rough bound of the form $|\Delta_i(h)(x)| \leq KC(\beta)n\log{n}$ for all $x\in \Omega$. In the next section we improve the bound for $x$ such that $m(x)$ is close to $m^{*}$ via a more delicate coupling argument.

\section{Coupling Argument}
\label{s:coupling}

\subsection{Coupling for Improved Bound on $\Delta_i(h)$}
  For $x,y \in \Omega$, we write $x \geq y$ iff $x^i \geq y^i$ for every $i \in [n]$. 
We recall the monotone coupling from Subsection~\ref{subsec:monotone_coupling}.  If the current states are $X$ and $Y$, we update the states to $X^{\prime}$ and $Y^{\prime}$ respectively as follows: we choose the same random index $I \sim \mathsf{unif} ([n])$. For all $j \neq I$, set $(X^{\prime})^j = X^j $ and $(Y^{\prime})^j = Y^j $. Generate an independent random variable $u_t \sim \mathrm{unif}([0,1])$. 
Set $(X^{\prime})^I$ (and $(Y^{\prime})^I$) to $1$ iff 
$u_t \leq \mu_I(1\rvert X^{(\sim i)})$ 
(and $u_t \leq \mu_I(1\rvert Y^{(\sim i)})$). 
For ferromagnetic Ising models when the update rule above is used, $X^{\prime} \geq Y^{\prime}$ almost surely if $X \geq Y$.  


We will shortly describe the coupling we use for the restricted Glauber dynamics, but we need to first record some useful properties of $g(s)=\tanh{\beta s}-s$ which follow from elementary calculus.
\begin{lemma}
\label{negative_slope_properties} Let $\beta>1$ and consider the function $g(s) = \tanh{\beta s} - s$ for $s \in [0,1]$. Denote by $m^*$ the strictly positive root of $g$. Then $g$ is concave, $g(0)=0$, $g^{\prime}(0) = \beta - 1 >0$, $g^{\prime}(m^{*}) := -\gamma^{*} < 0$, and also
\begin{enumerate}
\item For every $m > m^{*}$, $g^{\prime}(m) < -\gamma^{*}$ and
\item There are $s_1,s_2 \in (0,1)$ with $s_1 < s_2 < m^{*}$ and $g^{\prime}(s_2)< g^{\prime}(s_1) < -\frac{1}{2}\gamma^{*}$\,.
\end{enumerate}
\end{lemma} 
We fix values $s_1$ and $s_2$ as given in the lemma (see Figure~\ref{fig:magnetization_profile} to understand the significance of the various quantities defined above). The scalar $s$ indexes the values of magnetization. The restricted Glauber dynamics for the Curie-Weiss model contracts whenever the magnetization value is in the red region -- i.e., where the slope of $g(s)$ is negative. Lemma~\ref{expander_magnetisation_concentration} shows that under measure $\nu(\cdot)$, the magnetization concentrates in the blue region.
 
\begin{figure}
\centering
\includegraphics[scale=0.3]{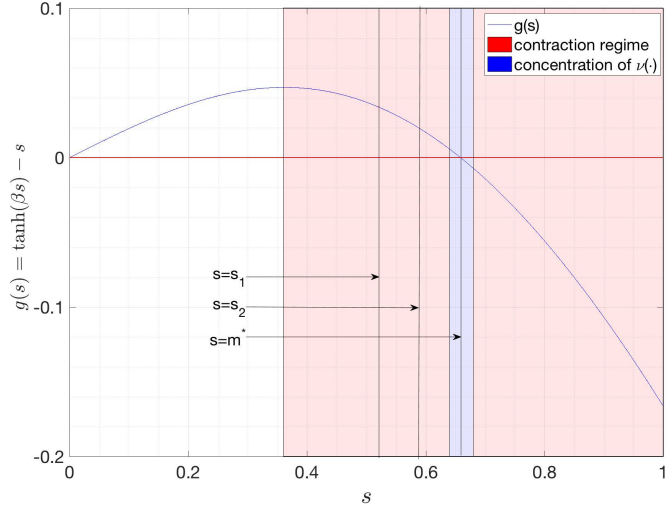}
\vspace{-5mm}
\caption{$g(s)$ for $\beta = 1.2$}
\label{fig:magnetization_profile}
\vspace{-3mm}
\end{figure}

Let the set $S_n := \{-1,-1+\frac{2}{n},\dots,+1\}$ (that is, the set of all possible values of $m(x)$). For any $s \in [-1,1]$, define $\langle s\rangle := \sup S_n\cap [-1,s]$.

\textbf{The Coupling:}
Let $x_0 \in \Omega^{+}$ be an arbitrary point such that $m(x_0) \geq \frac{2}{n}$. 
Consider two restricted Glauber chains $\hat{X}_t$ and $\hat{Y}_t$ for the Curie-Weiss model, with stationary measure $\mu^{+}$, such that $\hat{X}_0 = x_0^{(i,+)} \in \Omega^{+}$ and $\hat{Y}_0 = x_0^{(i,-)} \in \Omega^{+}$. 
We define $\tau_1 = \inf\{t: m(\hat{Y}_t) = \langle s_1\rangle\}$ and use the following coupling between $\hat{X}_t$ and $\hat{Y}_t$:
\begin{enumerate} 
\item If $m(x_0) \leq \langle s_2\rangle $, we couple them as in Remark~\ref{rem:coupling_existence}.
\item 
If $m(x_0) > \langle s_2 \rangle$ and $t \leq \tau_1$, monotone couple $\hat{X}_t$ and $\hat{Y}_t$. If $\hat{X}_{\tau_1} = \hat{Y}_{\tau_1}$, couple them so that $\hat{X}_t= \hat{Y}_t$ for $t > \tau_1$. Since $\hat{X}_0 \geq \hat{Y}_0$, the monotone coupling ensures that $\hat{X}_t \geq \hat{Y}_t$ for $t \leq \tau_1$. 
\item 
 If $\hat{X}_{\tau_1} \neq \hat{Y}_{\tau_1}$, then for $t> \tau_1$, we couple them as in Remark~\ref{rem:coupling_existence}.
\end{enumerate}

Suppose that $x_0 \in \Omega^{+}$ is such that $m(x_0) > \langle s_2 \rangle$. The coupling above is constructed to give a better bound on $|\Delta_i(h)(x_0)|$ than in Lemma~\ref{naive_bound}. The intuition behind it is that whenever $m(\hat{X}_t) \geq \langle s_1\rangle$ and $m(\hat{Y}_t) \geq \langle s_1\rangle$ (that is, when $t \leq \tau_1$), the chains are contracting under the monotone coupling. This is shown in Lemma~\ref{super_martingale_contraction} and used in Lemma~\ref{contraction_bound_lemma} to bound $|\Delta_i(h)(x_0)|$ in terms of $\rho^{K}+ \mathbb{P}(\tau_1 < K)$ (where $\rho = 1-\Theta(\tfrac{1}{n})$ is the contraction coefficient and $K$ is any integer). This proof is a generalization of the proof of Theorem~\ref{t:contractIsing}. 

To use this bound we need to show that $\mathbb{P}(\tau_1 < K)$ is small, i.e.
the walk usually takes a long time to hit $\langle s_1\rangle$. This is shown in Lemma~\ref{lem:escape_time_bound} as a consequence of $m(\hat{Y}_t)$ being a birth-death process with positive drift when it is between $\langle s_1\rangle$ and $\langle s_2 \rangle$.



 Define
\begin{equation}
\tau_{\mathrm{coup}} = \begin{cases}
0 &\text{ if $\hat{X}_{\tau_1} = \hat{Y}_{\tau_1}$} \\
\inf\{t: \hat{X}_t = \hat{Y}_t\} - \tau_1 &\text{ otherwise}\,.
\end{cases}
\end{equation}

\begin{lemma}\label{super_martingale_contraction}
Let $x_0\in \Omega$ be such that $m(x_0) \geq s_2 + \frac{2}{n} $. Let $f$ be symmetric and $\frac{1}{n}$-Lipschitz in each coordinate. Let $\gamma^{*} >0$ be as in Lemma~\ref{negative_slope_properties}. Define the chains $\hat{Y}_t$ and $\hat{X}_t$ as defined above, and let $\rho := \big(1-\frac{\gamma^{*}(n-1)}{2n^2}\big)$. 
Then, the following hold:
\begin{enumerate}
\item
$\mathbb{E}\Big[\big|f(\hat{X}_t) - f(\hat{Y}_t)\big|\ind_{t\leq \tau_1}\Big| \hat{X}_0 = x_0^{(i,+)}, \hat{Y}_0 = x_0^{(i,-)}\Big]\leq \frac{1}{n}\rho^t$\vspace{1mm}
\item
$\mathbb{P}(\hat{X}_{\tau_1} \neq \hat{Y}_{\tau_1}|\tau_1 \geq K) \leq  \frac{\rho^{K}}{\mathbb{P}(\tau_1 \geq K)} $.
\end{enumerate} 
\end{lemma}
\begin{proof}
Let $ 1 \leq t \leq \tau_1$. By the Lipschitz property of $f$ and monotone coupling between the chains, 
\begin{align}
|f(\hat{X}_t) - f(\hat{Y}_t)| &\leq  \frac{1}{n}\sum_{i=1}^{n}\ind(\hat{X}_t^{i}\neq \hat{Y}_t^{i}) \\
&= \frac{1}{2n}\sum_{i=1}^{n}|\hat{X}^{i}_t - \hat{Y}^{i}_t| \nonumber \\
&= \frac{1}{2n}\sum_{i=1}^{n}\hat{X}^{i}_t - \hat{Y}^{i}_t \nonumber \\
&= \frac{1}{2}\big(m(\hat{X}_t) - m(\hat{Y}_t)\big)\,. \label{lipschitz_property}
\end{align}
Let $m_i := \frac1n{\sum_{j \neq i}x^j}$ so that 
$$
\mu_i(1|x^{(\sim i)}) = \frac12 + \frac12{\tanh(\beta m_i)}\,.
$$
Note that $\sum_{i=1}^n m_i = (n-1)m$. By monotonicity of the coupling and definition of $\tau_1$, $m_i(\hat{X}_{t-1}) \geq m_i(\hat{Y}_{t-1}) \geq s_1$ almost surely, and we assume in what follows that $x_{t-1}$ and $y_{t-1}$ satisfy $m_i(x_{t-1}) \geq m_i(y_{t-1}) \geq s_1$. Conditioning on whether or not an update occurs at a location in which $x_{t-1}$ and $y_{t-1}$ differ, we obtain
\begin{align}
&\mathbb{E}\Big[m(\hat{X}_t) - m(\hat{Y}_t)|\hat{X}_{t-1} = x_{t-1}, \hat{Y}_{t-1} = y_{t-1}\Big] \nonumber \\ 
&= m(x_{t-1})-m(y_{t-1}) - \frac{1}{n^2} \sum_{i=1}^n (x_{t-1}^i -y_{t-1}^i)
\nonumber\\ 
&\quad+ \frac{1}{n^2}\sum_{i=1}^{n}\big(\tanh{(\beta m_i(x_{t-1})})-\tanh{(\beta m_i(y_{t-1}))}\big) \nonumber\\
&= m(x_{t-1})-m(y_{t-1}) -\frac{1}{n(n-1)}\sum_{i=1}^{n}\big(m_i(x_{t-1}) - m_i(y_{t-1})\big)\nonumber\\
&\quad+\frac{1}{n^2}\sum_{i=1}^{n}\big(\tanh(\beta m_i(x_{t-1}))-\tanh{(\beta m_i(y_{t-1}))}\big)\nonumber \\
&\leq  m(x_{t-1})-m(y_{t-1}) +\frac{1}{n^2}\sum_{i=1}^n g\big(m_i(x_{t-1})\big) - g\big(m_i(y_{t-1})\big)\nonumber\\
&\leq m(x_{t-1}) - m(y_{t-1}) -\frac{\gamma^{*}}{2n^2} \sum_{i=1}^{n}\left(m_i(x_{t-1}) - m_i(y_{t-1})\right) \nonumber\\
&= \Big(1-\frac{\gamma^{*}(n-1)}{2n^2}\Big)\big(m(x_{t-1}) - m(y_{t-1})\big) \nonumber\\
&= \rho\big(m(x_{t-1}) - m(y_{t-1})\big)\,.\label{contraction_curie_weiss}
\end{align}
Here we have used the properties of $g$ stated in Lemma~\ref{negative_slope_properties}. Therefore, for $t \leq \tau_1$, $M_t = \rho^{-t}(m(\hat{X}_{t}) - m(\hat{Y}_{t}))$ is a positive super-martingale with respect to the filtration $\mathcal{F}_t = \sigma\big(\hat{X}_0,\hat{Y}_0,\hat{X}_1,\hat{Y}_1,\dots,\hat{X}_t,\hat{Y}_t\big)$ and $\tau_1$ is a stopping time. By the Optional Stopping Theorem, we conclude that 
\begin{align*}
\frac{2}{n} &= \mathbb{E}[M_0]\\
&\geq \mathbb{E}\left[M_{t\wedge\tau_1}\right] \\
&\geq \mathbb{E}\big[\rho^{-t}(m(\hat{X}_t)- m(\hat{Y}_t)) \ind_{t \leq \tau_1}\big]\,.
\end{align*}
Thus, $\mathbb{E}\big[(m(\hat{X}_t) - m(\hat{Y}_t))\ind_{t \leq \tau_1}\big] \leq \frac{2\rho^{t}}{n}$. We use~\eqref{lipschitz_property} to complete the proof of the first part of the lemma.

Turning to the second part, using the fact that $\rho <1$ gives
\begin{align}
\frac{2}{n} &= \mathbb{E}[M_0]\nonumber \\
&\geq \mathbb{E}\left[M_{\tau_1}\right] \nonumber \\
&= \mathbb{E}\left[\rho^{-\tau_1}(m(\hat{X}_{\tau_1})-m(\hat{Y}_{\tau_1})) \right]\nonumber \\
&\geq \mathbb{E}\left[\rho^{-K}(m(\hat{X}_{\tau_1})-m(\hat{Y}_{\tau_1}))|\tau_1 \geq K\right]\mathbb{P}(\tau_1 \geq K) \label{expectation_bound}\,.
\end{align}
By monotone coupling, we know that $\hat{X}_{\tau_1} \neq \hat{Y}_{\tau_1}$ iff $m(\hat{X}_{\tau_1})-m(\hat{Y}_{\tau_1}) \geq \frac{2}{n}$. Therefore, using Markov's inequality and~\eqref{expectation_bound} we conclude that
\begin{align*}
\mathbb{P}(\hat{X}_{\tau_1} \neq \hat{Y}_{\tau_1}|\tau_1 \geq K) &= \mathbb{P}\Big(m(\hat{X}_{\tau_1})-m(\hat{Y}_{\tau_1}) \geq \frac{2}{n}\Big|\tau_1 \geq K\Big) \\
&\leq \frac{n\cdot \mathbb{E}\big[\big(m(\hat{X}_{\tau_1})-m(\hat{Y}_{\tau_1})\big)|\tau_1 \geq K\big]}{2} \\
&\leq 
\frac{\rho^{K}}{\mathbb{P}(\tau_1 \geq K)}\,.&\qedhere
\end{align*} 
\end{proof}

\begin{lemma}\label{contraction_bound_lemma}
Let $x_0\in \Omega$ be such that $m(x_0) \geq s_2 + \frac{2}{n} $. Let $\hat{X}_t$, $\hat{Y}_t$, $f$, $\rho$ and $h$ be as defined above. Then for every $K \in \mathbb{N}$,
$$|\Delta_i(h)(x_0)| \leq \frac{1}{n}\frac{1}{1-\rho} + C(\beta)n\log{n}\left[\rho^K + \mathbb{P}(\tau_1 < K) \right]\,.$$
\end{lemma}

\begin{proof}
For the sake of brevity, only in this proof, we implicity assume the conditioning $\hat{X}_0 = x_0^{(i,+)}$ and $\hat{Y}_0 = x_0^{(i,-)}$ whenever the expectation operator is used. Expanding the principal solution to the Poisson equation yields
\begin{align}
|\Delta_i(h)(x_0)| &= \bigg|\sum_{t=0}^{\infty}\mathbb{E}\left[\big(f(\hat{X}_t) - f(\hat{Y}_t)\big)\right]\bigg| \nonumber \\
&\leq \sum_{t=0}^{\infty}\mathbb{E}\left[\big|f(\hat{X}_t)-f(\hat{Y}_t)\big|\right]\nonumber \\
&= \sum_{t=0}^{\infty}\mathbb{E}\left[\big|f(\hat{X}_t)-f(\hat{Y}_t)\big|(\ind_{t \leq \tau_1}+\ind_{t>\tau_1})\right]\nonumber\\
&\leq \sum_{t=0}^{\infty} \frac{\rho^t}{n} + \sum_{t=0}^{\infty}\mathbb{E}\left[\big|f(\hat{X}_t)-f(\hat{Y}_t)\big|\ind_{t>\tau_1} \right] \nonumber \\
&= \frac{1}{n}\frac{1}{1-\rho} + \sum_{t=0}^{\infty}\mathbb{E}\left[\big|f(\hat{X}_t)-f(\hat{Y}_t)\big|\ind_{t>\tau_1}\right]\,.
\label{exponential_coupling}
\end{align}
Here we have used Lemma~\ref{super_martingale_contraction} in the second to last step.
By definition of the coupling, if $\hat{X}_{\tau_1} = \hat{Y}_{\tau_1}$, then $f(\hat{X}_t) - f(\hat{Y}_t) = 0$ for all $t > \tau_1$. Further, $|f(\hat{X}_t)-f(\hat{Y}_t)| \leq \ind_{t \leq \tau_{\mathrm{coup}}+\tau_1}$ (since $|f(x) - f(y)| \leq 1$).  
Given $K\in \mathbb{N}$, we conclude that
\begin{align}
&\sum_{t=0}^{\infty}\mathbb{E}\big[|f(\hat{X}_t)-f(\hat{Y}_t)\big|\ind_{t>\tau_1}\big] \nonumber
\\&\leq 
\mathbb{E}\left[\tau_{\mathrm{coup}} \right] \nonumber \\
&= \sum_{x,y \in \Omega^{+}}\mathbb{E}\big[\tau_{\mathrm{coup}}\big|  \hat{X}_{\tau_1}=x, \hat{Y}_{\tau_1}=y\big]\cdot \mathbb{P}(\hat{X}_{\tau_1} = x, \hat{Y}_{\tau_1} = y)\nonumber 
\\&\leq C(\beta)n\log{n}\sum_{x,y \in \Omega^{+}}\ind_{x\neq y}\mathbb{P}(\hat{X}_{\tau_1} = x, \hat{Y}_{\tau_1} = y)\nonumber \\ 
&= C(\beta)n\log{n} \mathbb{P}(\hat{X}_{\tau_1}\neq \hat{Y}_{\tau_1}) \nonumber \\
&= C(\beta)n\log{n}\mathbb{P}(\hat{X}_{\tau_1}\neq \hat{Y}_{\tau_1} | \tau_1 \geq K)\mathbb{P}(\tau_1 \geq K)\nonumber \\
&\quad + C(\beta)n\log{n}\mathbb{P}(\hat{X}_{\tau_1}\neq \hat{Y}_{\tau_1} | \tau_1 < K)\mathbb{P}(\tau_1 < K)\nonumber \\
&\leq C(\beta)n\log{n}\left[\mathbb{P}(\hat{X}_{\tau_1}\neq \hat{Y}_{\tau_1} | \tau_1 \geq K)\mathbb{P}(\tau_1 \geq K) + \mathbb{P}(\tau_1 < K)\right]\nonumber \\
&\leq \label{non_exponential_coupling}
C(\beta)n\log{n}\left[\rho^K + \mathbb{P}(\tau_1 < K) \right]\,. 
\end{align}
Here we have used Theorem~\ref{t:Levin} in the second inequality and Lemma~\ref{super_martingale_contraction} in the last inequality.
By~\eqref{exponential_coupling} and~\eqref{non_exponential_coupling}, we conclude the result.
\end{proof}

Lemma~\ref{contraction_bound_lemma} bounds $|\Delta_i(h)|$ in terms of $\mathbb{P}(\tau_1 < K)$. We upper bound this probability in the following lemma.

\begin{lemma} 
\label{lem:escape_time_bound}
Let $x_0 \in \Omega$ be such that $m(x_0)\geq \langle s_2\rangle + \frac{2}{n}$.
For every integer $K$,
$$\mathbb{P}(\tau_1 < K) \leq K^2 \exp{\left(-c_1(\beta)n\right)}\,.$$
Here $c_1(\beta) >0$ is a constant that depends only on $\beta$.
\end{lemma}

The proof, which we defer to Appendix~\ref{subsec:escape_time_bound_proof}, is by coupling the magnetization chain to an appropriate birth-death chain and using hitting time results for birth-death chains.

\begin{theorem}\label{contraction_bound}
If $m(x_0) \geq \langle s_2\rangle + \frac{2}{n}$, then there are constants $c$ and $c'$ depending only on $\beta$ such that
$$|\Delta_i(h)(x_0)| \leq \frac{4}{\gamma^{*}}\big(1+ c \cdot \exp(-c' n)\big)\,.$$ 
\end{theorem}
\begin{proof} 
 By Lemma~\ref{contraction_bound_lemma}, we have for every positive integer $K$,
$$|\Delta_i(h)(x_0)| \leq \frac{1}{n}\frac{1}{1-\rho} + C(\beta)n\log{n}\left[\rho^K + \mathbb{P}(\tau_1 < K) \right]\,.$$ 
Clearly, for $n\geq 2$, $$\frac{1}{n}\frac{1}{1-\rho} \leq \frac{4}{\gamma^{*}}\,.
$$
By Lemma~\ref{lem:escape_time_bound}, $\mathbb{P}(\tau_1 < K) \leq K^2 \exp{(-c_1(\beta)n)} $, and we take $K \geq Cn^2$.
\end{proof}

We are now ready to prove Theorem~\ref{main_application}.
\section{Proof of Theorem~\ref{main_application}}
\label{s:IsingProof}
We use all the notation developed in Section~\ref{s:ProofIdeas}. 
 Let $h$ be the solution to the Poisson equation
$(I-P)h = f - \mathbb{E}_{\mu}f$ with $f$ defined in \eqref{e:f} at the beginning of Section~\ref{s:ProofIdeas}.
It follows by Theorem~\ref{main_theorem} and Lemma~\ref{spectral_bound} to show that 
\begin{equation}
|\mathbb{E}_\mu f - \mathbb{E}_\nu f| \leq \|\tfrac{\beta}{n}A-\tfrac{\beta}{d}B\|_2 \mathbb{E}_{\nu}\frac{\|v_{\Delta(h)}\|_2}{2\sqrt{n}}\,,
\label{main_bound2}
\end{equation}
where $v_{\Delta(h)} := (\Delta_1(h),..,\Delta_n(h))^{\intercal}$. By Jensen's inequality,
\begin{align}
\mathbb{E}_{\nu} \|v_{\Delta(h)}\| &= \mathbb{E}_{\nu} \sqrt{\sum_{i=1}^n \Delta_i(h)^2} \leq \sqrt{\sum_{i=1}^n \mathbb{E}_{\nu}\Delta_i(h)^2}\,. \label{sqrt_jensen}
\end{align}
Now, using Lemmas~\ref{symmetric_solution} and~\ref{naive_bound} and Theorem~\ref{contraction_bound} we conclude 
\begin{equation*}
|\Delta_i(h)(x)| \leq \begin{cases}
 \frac{4}{\gamma^{*}}(1+o_n(1)) &\text{if $|m(x)| \geq \langle s_2\rangle + 2/n$}\\
 C(\beta)n\log{n} &\text{otherwise}\,.
\end{cases}
\end{equation*}
We take $0< \delta(\beta) < m^{*} - \langle s_2\rangle - \frac{2}{n}$ to be dependent only on $\beta$. By Lemma~\ref{expander_magnetisation_concentration} there exists $\epsilon_0$ such that if $\epsilon < \epsilon_0$, then for some $c(\delta) > 0$
$$
\nu^{+}(|m-m^{*}| > \delta) \leq e^{-c(\delta)n}\,.
$$
By Lemma~\ref{symmetric_solution}, $\Delta_i(h)^2$ is a symmetric function of  $x$. Therefore, 
\begin{align}
\mathbb{E}_{\nu}\Delta_i(h)^2 &= \mathbb{E}_{\nu^+}\Delta_i(h)^2 \nonumber \\
&\leq \frac{16}{(\gamma^{*})^2}(1+o(1)) + v^{+}(|m-m^{*}| > \delta) C(\beta)^2n^2\log^2{n}\nonumber \\
&\leq \frac{16}{(\gamma^{*})^2}(1+o(1)) + e^{-c(\delta)n} C(\beta)^2n^2\log^2{n}\nonumber \\
&= \frac{16}{(\gamma^{*})^2}(1+o(1))\,. 
\label{expectation_bound_1}
\end{align}

We note that by picking $C_{R} = \mathrm{sgn}(\rho^{(k)}[R] - \tilde{\rho}^{(k)}[R])$, we obtain that
$$\frac{1}{{{n}\choose{k}}} \sum_{\substack{R \subset [n] \\ |R| = k}}|\rho^{(k)}[R] - \tilde{\rho}^{(k)}[R]|  = 2k|\mathbb{E}_\mu f - \mathbb{E}_\nu f|\,.$$
The equation above along with~\eqref{sqrt_jensen},~\eqref{expectation_bound_1}, and~\eqref{main_bound2}, implies
\begin{align*}
\frac{1}{{{n}\choose{k}}} \sum_{\substack{R \subset [n] \\ |R| = k}}|\rho^{(k)}[R] - \tilde{\rho}^{(k)}[R]|  &= 2k|\mathbb{E}_\mu f - \mathbb{E}_\nu f| \\
&\leq 2k\|\tfrac{\beta}{n}A-\tfrac{\beta}{d}B\|_2 \mathbb{E}_{\nu}\frac{\|v_{\Delta(h)}\|_2}{2\sqrt{n}} \\
&\leq \frac{k}{\sqrt{n}}\beta \left(\epsilon + \tfrac{1}{n}\right) \sqrt{\sum_{i}\mathbb{E}_{\nu}(\Delta_i(h))^2} \\
&\leq 4\frac{k\beta}{\gamma^{*}}(1+o_n(1))\left(\epsilon +\tfrac{1}{n}\right)\,,
\end{align*} 
which completes the proof. \hfill \qed
\section{Comparison to Naive Bounds}
\label{s:naive}
Using the symmetry inherent in the Curie-Weiss model, we sketch another method to obtain an inequality similar (but much weaker) to the one in Theorem~\ref{main_application}. We don't give the proofs of the results below. All of them can be proved using definitions and standard techniques.
Let $D_{\mathsf{SKL}}(\mu; \nu) = D_{\mathsf{KL}}(\mu\| \nu) + D_{\mathsf{KL}}(\nu\| \mu)  $ denote the symmetric KL-divergence between measures $\mu$ and $\nu$.
\begin{lemma}
 $D_{\mathsf{KL}} (\nu || \mu)\leq D_{\mathsf{SKL}}(\mu;\nu) \leq n\|\tfrac{\beta}{n}A-\tfrac{\beta}{d}B\|_2 $
\label{KL_bound}
\end{lemma}
Let $X \sim \mu$ and $X^{\prime} \sim \mu$ such that they are independent of each other. Define $m_2(X,X^{\prime}) := \frac{1}{n} \sum_{i=1}^{n}X^i(X^{\prime})^i$
\begin{lemma}
For the Curie Weiss model at any fixed temperature,
$$\log \mathbb{E}_\mu \exp{ \lambda(m^2-(m^{*})^2)} \leq O(\log{n}) + 
\frac{C_1(\beta) \lambda^{2}}{2n}$$ and
$$\log\mathbb{E}_{\mu\otimes\mu}\exp{\lambda(m_2^2 -(m^*)^4)} \leq  O(\log{n})+ \frac{C_2(\beta)}{2n}\lambda^2\,.$$
 Here $C_1(\beta)$ and $C_2(\beta)$ are positive constants that depend only on $\beta$.
 \label{sub_gaussian}
\end{lemma}

Consider the set of probability distributions over $\Omega \times \Omega$, $\mathcal{S} = \{M: M \ll \mu\otimes \mu\}$. Let $f: \Omega \times \Omega \to \mathbb{R}$ be defined by $f(x,x^{\prime}) = m_2^{2} - (m^*)^4$. By Gibbs' variational principle, 
$$\log{\mathbb{E}_{\mu\otimes \mu}\exp{\lambda f}} = \sup_{M \in \mathcal{S}} \lambda\mathbb{E}_M f - D_{\mathsf{KL}}(M|| \mu\otimes \mu)\,.$$
Taking $M = \nu \otimes \nu$ (whence $D_{\mathsf{KL}}(\nu \otimes \nu|| \mu\otimes \mu) = 2D_{\mathsf{KL}}(\nu ||  \mu)$) and using Lemma~\ref{sub_gaussian}, we conclude that:
$$\lambda\mathbb{E}_{\nu \otimes \nu} f - 2D(\nu||\mu) \leq  C\log{n}+ \frac{C_2}{2n}\lambda^2\,.$$ 
Letting $\lambda = \frac{n\mathbb{E}_{\nu \otimes \nu} f}{C_2}$, we conclude that
$$|\mathbb{E}_{\nu \otimes \nu} f | = O\left( \sqrt{\frac{\log{n}}{n}} +\sqrt{ \frac{D_{\mathsf{KL}}(\nu ||  \mu)}{n}}\right)\,.$$
and taking $M = \mu \otimes \mu$  (whence $D_{\mathsf{KL}}(\mu \otimes \mu|| \mu\otimes \mu) = 0$) we conclude that
$$|\mathbb{E}_{\mu \otimes \mu} f | = O\left(\sqrt{\frac{\log{n}}{n}}\right) \,.$$
Therefore,
\begin{equation}
|\mathbb{E}_{\mu \otimes \mu} f - \mathbb{E}_{\nu \otimes \nu} f | = O\left( \sqrt{\frac{\log{n}}{n}} +\sqrt{ \frac{D_{\mathsf{KL}}(\nu ||  \mu)}{n}}\right)\,.
\label{expectation_compare_1}
\end{equation}
By similar considerations, taking $g(x) = m^2 - (m^*)^2$, we conclude that
\begin{equation}
\left|\mathbb{E}_{\nu }[ g]- \mathbb{E}_{\mu} [g]\right| = O\left( \sqrt{\frac{\log{n}}{n}} +\sqrt{ \frac{D_{\mathsf{KL}}(\nu ||  \mu)}{n}}\right)\,.
\label{expectation_compare_2}
\end{equation}

For the Curie-Weiss model, by symmetry, $\rho_{ij} = \rho$ (the same for all $i\neq j$). 
Clearly,
$$\mathbb{E}_\mu m^2 = \frac{1}{n} + \frac{2}{n^2} \sum_{i\neq j} \rho$$
$$\mathbb{E}_\nu m^2 = \frac{1}{n} + \frac{2}{n^2} \sum_{i\neq j} \tilde{\rho}_{ij}$$
$$\mathbb{E}_{\mu\otimes \mu} m_2^2 = \frac{1}{n} +\frac{2}{n^2} \sum_{i\neq j} \rho^2$$
$$\mathbb{E}_{\nu \otimes \nu} m_2^2 = \frac{1}{n} + \frac{2}{n^2} \sum_{i\neq j} \tilde{\rho}^2_{ij}\,.$$
Therefore,
\begin{align*}
 \sum_{i\neq j} (\rho_{ij} -\tilde{\rho}_{ij})^2
 &= \sum_{i\neq j} \tilde{\rho}_{ij}^2 +\rho^2 - 2\rho(\sum_{i\neq j} \tilde{\rho}_{ij}) \\
 &= \sum_{i\neq j} \tilde{\rho}_{ij}^2 +\rho^2 - 2\rho\left(\sum_{i\neq j} \rho + \frac{n^2}{2}(\mathbb{E}_\nu m^2 - \mathbb{E}_\mu m^2)\right) \\
 &= \sum_{i \neq j} \tilde{\rho}_{ij}^2 -\rho^2 - n^2 (\mathbb{E}_\nu m^2 - \mathbb{E}_\mu m^2)) \\
 &= \frac{n^2}{2}\left(\mathbb{E}_{\nu \otimes \nu} m_2^2 - \mathbb{E}_{\mu \otimes \mu} m_2^2\right)-n^2 (\mathbb{E}_\nu m^2 - \mathbb{E}_\mu m^2)\\
&\leq n^2 |\mathbb{E}_{\mu \otimes \mu} f - \mathbb{E}_{\nu \otimes \nu} f | + n^2|\mathbb{E}_\mu g - \mathbb{E}_\nu g|\,.
\end{align*}
Using the equation above and Equations~\ref{expectation_compare_1} and~\ref{expectation_compare_2} and Lemma~\ref{KL_bound} we conclude that
\begin{equation}
 \frac{1}{{{n}\choose{2}}}\sum_{i\neq j} (\rho_{ij} -\tilde{\rho}_{ij})^2 \leq O\left( \sqrt{\frac{\log{n}}{n}} +\sqrt{\|\tfrac{\beta}{n}A-\tfrac{\beta}{d}B\|_2 }\right)\,.
\end{equation}
When $\epsilon = o(\tfrac{\log{n}}{n})$, the equation above reduces to:
$$ \frac{1}{{{n}\choose{2}}}\sum_{i\neq j} (\rho_{ij} -\tilde{\rho}_{ij})^2 \leq O\left( \sqrt{\epsilon} \right)\,.$$
This is similar to the result in Theorem~\ref{main_application} but weaker by a 4th power.

\section*{Acknowledgment} GB is grateful to Andrea Montanari and Devavrat Shah for discussions on related topics. Also we thank Gesine Reinert and Nathan Ross for exchanging manuscripts with us.

\bibliographystyle{imsart-number}
\bibliography{references}
\appendix
\section{Proofs of Lemmas}
\subsection{Proof of Lemma~\ref{expander_magnetisation_concentration}}
\begin{proof}
The proof follows the standard proof of concentration of magnetization for the Curie-Weiss model, but with a slight modification to account for the spectral approximation. Let $\gamma \in S_n  := \{m(x) : x \in \Omega\}$. By $M_\gamma$ we denote the set $\{x \in \Omega : m(x) = \gamma\}$. Note that $|M_{\gamma}| = {{n}\choose{n\frac{1+\gamma}{2}}}$ and $|S_n| = n+1$.

We define
$$Z_\gamma = \sum_{x \in M_\gamma}  e^{\frac{\beta}{2d}x^{\intercal}Bx}$$ 
and 
\begin{align*}
Z = \sum_{x \in \Omega} e^{\frac{\beta}{2d}x^{\intercal}Bx} 
= \sum_{\gamma \in S_n} Z_{\gamma}\,,
\end{align*}
and for any $U \subset S_n$, 
\begin{align*}
Z_U = \sum_{x: m(x) \in U} e^{\frac{\beta}{2d}x^{\intercal}Bx} = \sum_{\gamma \in U} Z_{\gamma}\,.
\end{align*}

Clearly, $$ \frac{\beta}{2n}x^{\intercal}Ax - \frac{1}{2}\|\tfrac{\beta}{n}A-\tfrac{\beta}{d}B\|x^{\intercal}x\leq \frac{\beta}{2d}x^{\intercal}Bx \leq \frac{\beta}{2n}x^{\intercal}Ax + \frac{1}{2}\|\tfrac{\beta}{n}A-\tfrac{\beta}{d}B\|x^{\intercal}x\,.$$\
Using the identities $\frac{\beta}{2n}x^{\intercal}Ax  = \frac{\beta n}{2}(m^2-\frac{1}{n})$ and $x^{\intercal}x = n$, as well as~\eqref{norm_bound}, we conclude that 
$$\frac{\beta n}{2} (m^2 - \epsilon - \frac{2}{n}) \leq \frac{d\beta}{2n}x^{\intercal}Bx \leq \frac{\beta n}{2} (m^2 + \epsilon )\,,$$ which implies that
$$  {{n}\choose{n\frac{1+\gamma}{2}}}\exp{\frac{\beta n}{2} (\gamma^2 - \epsilon - \frac{2}{n})} \leq Z_{\gamma} \leq {{n}\choose{n\frac{1+\gamma}{2}}}  \exp{\frac{\beta n}{2} (\gamma^2 + \epsilon )}\,.$$
Let $H: [0,1] \to \mathbb{R}$ be the binary Shannon entropy. Stirling's approximation gives that 
 $$ \frac{e^{nH(\frac{1+\gamma}{2})}}{\sqrt{2n}} \leq {{n}\choose{n\frac{1+\gamma}{2}}} \leq e^{nH(\frac{1+\gamma}{2})}$$
 and we conclude that
 $$  \frac{\beta}{2}\gamma^2  + H\left(\frac{1+\gamma}{2}\right)- \frac{\beta}{2}\epsilon + O\left(\frac{\log{n}}{n}\right) \leq \frac{\log{Z_{\gamma}}}{n} \leq \frac{\beta}{2} \gamma^2  + H\left(\frac{1+\gamma}{2}\right) + \frac{\beta}{2}\epsilon + O\left(\frac{\log{n}}{n}\right)\,.$$
Using the equation above, for any $U \subset S_n$
 \begin{align*}
  \frac{\log{Z_U}}{n} &\leq \log{\left(|U| \max_{\gamma \in U}Z_{\gamma}\right)}\\
  &= \frac{\log{|U|}}{n} + \max_{\gamma \in U}\frac{\log{Z_\gamma}}{n} \\
  &\leq \max_{\gamma \in U} \left[\frac{\beta}{2} \gamma^2  + H\left(\frac{1+\gamma}{2}\right) \right]+ \frac{\beta}{2}\epsilon + O\left(\frac{\log{n}}{n}\right)\,.
 \end{align*} 
 Here we have used the fact that $|U| \leq |S_n| = n+1$. Similarly,
 \begin{align*}
 \frac{\log{Z}}{n} &= \log{\frac{\sum_{\gamma \in S_n} Z_{\gamma}}{n}}\\
 &\geq \max_{\gamma \in S_n}\frac{\log{Z_{\gamma}}}{n} \\
 &\geq  \max_{\gamma \in S_n} \left[\frac{\beta}{2} \gamma^2  + H\left(\frac{1+\gamma}{2}\right) \right]- \frac{\beta}{2}\epsilon + O\left(\frac{\log{n}}{n}\right)\,.
 \end{align*}
 
 Define $U_\delta = S_n\setminus ([m^*-\delta,m^*+\delta]\cup[-m^*-\delta,-m^*+\delta])$ and $V_\delta = [0,1]\setminus ([m^*-\delta,m^*+\delta]\cup[-m^*-\delta,-m^*+\delta])$.
 Clearly, $$ \nu\left(\{|m(x)-m^{*}| > \delta\}\cap \{|m(x)+m^{*}| > \delta\}\right) = \nu\left(m(x) \in U_\delta\right) $$ is the probability to be bounded. We get
\begin{align} 
&\frac{\log{ \nu\left(m(x) \in U_\delta\right)}}{n} 
\notag \\&= \frac{\log{Z_{U_{\delta}}}}{n} - \frac{\log{Z}}{n}\nonumber \\
&\leq \max_{\gamma \in U_\delta} \left[\frac{\beta}{2} \gamma^2  + H\left(\frac{1+\gamma}{2}\right) \right] - \max_{\gamma \in S_n} \left[\frac{\beta}{2} \gamma^2  + H\left(\frac{1+\gamma}{2}\right) \right] + \beta\epsilon + O\left(\frac{\log{n}}{n}\right)\nonumber \\
&= \sup_{\gamma \in V_\delta} \left[\frac{\beta}{2} \gamma^2  + H\left(\frac{1+\gamma}{2}\right) \right] - \sup_{\gamma \in [0,1]} \left[\frac{\beta}{2} \gamma^2  + H\left(\frac{1+\gamma}{2}\right) \right] + \beta\epsilon + O\left(\frac{\log{n}}{n}\right)\,. \label{exponential_upper_bound}
 \end{align}
 Here we have used the properties of $H(\cdot)$ to show that $$\sup_{\gamma \in V_\delta} \Big[\frac{\beta}{2} \gamma^2+ H\Big(\frac{1+\gamma}{2}\Big) \Big] = \max_{\gamma \in U_\delta} \Big[\frac{\beta}{2} \gamma^2  + H\Big(\frac{1+\gamma}{2}\Big) \Big] + O\Big(\frac{\log{n}}{n}\Big) $$ and $$\sup_{\gamma \in [0,1]} \Big[\frac{\beta}{2} \gamma^2+ H\left(\frac{1+\gamma}{2}\right) \Big] = \max_{\gamma \in S_n} \Big[\frac{\beta}{2} \gamma^2  + H\Big(\frac{1+\gamma}{2}\Big) \Big] + O\Big(\frac{\log{n}}{n}\Big) \,. $$
 
 It can be shown by simple calculus that for $\beta > 1$, the function $\frac{\beta}{2} \gamma^2  + H\left(\frac{1+\gamma}{2}\right) $ has (all of) its global maxima at $m^*$ and $-m^*$. Since $V_\delta = [0,1]\setminus ([m^*-\delta,m^*+\delta]\cup[-m^*-\delta,-m^*+\delta])$, using the continuity of the function, we conclude that for some $c_0(\delta) > 0$,
 $$\sup_{\gamma \in V_\delta} \left[\frac{\beta}{2} \gamma^2  + H\left(\frac{1+\gamma}{2}\right) \right] - \sup_{\gamma \in [0,1]} \left[\frac{\beta}{2} \gamma^2  + H\left(\frac{1+\gamma}{2}\right) \right] < -c_0(\delta) < 0 \,.$$
Choosing $\epsilon$ small enough so that $\epsilon\beta < \frac{c_0(\delta)}{2}$, and using Equation~\eqref{exponential_upper_bound}, we conclude that
$$\nu\left(\{|m(x)-m^{*}| > \delta\}\cap \{|m(x)+m^{*}| > \delta\}\right)  \leq \exp{\left(-\frac{c_0(\delta)n}{2} + O(\log{n})\right)}$$
and the statement of the lemma follows.
\end{proof}

\subsection{Proof of Lemma~\ref{lem:escape_time_bound}}
\label{subsec:escape_time_bound_proof}
For the Curie-Weiss model, one can check that the Glauber dynamics also induces a Markov chain over the magnetization. For $m \in (0,1)$ the probability that $m \to m -\frac{2}{n}$ is $$\left(\frac{1+m}{2}\right)\left(\frac{1- \tanh{(\beta m+\frac{\beta}{n})}}{2}\right) =: p_{-}(m)$$ and probability that $m \to m + \frac{2}{n}$ is $$\left(\frac{1-m}{2}\right)\left(\frac{1+ \tanh{(\beta m-\frac{\beta}{n})}}{2}\right) =: p_{+}(m)\,.$$ At any step, this chain can only change the value of magnetization by $\frac{2}{n}$. By hypothesis, we start the restricted Glauber dynamics chain such that $\hat{Y}_0 = x_0^{(i,-)}$ with $m(x_0)\geq \langle s_2\rangle + \frac2n$. Therefore, $m(\hat{Y}_0) \geq \langle s_2 \rangle$. Recall that, by definition of $\tau_1$, $m(\hat{Y}_{\tau_1}) = \langle s_1\rangle$. Clearly, there exists $t < \tau_1$ such that $m(\hat{Y}_{t}) = \langle s_2\rangle$. That is, to reach a state with magnetization of $\langle s_1\rangle$, the chain must first hit a state with magnetization $\langle s_2 \rangle$. Therefore, $\mathbb{P}(\tau_1 < K| \hat{Y}_0 = x_0^{(i,-)})$ is maximized when $m(\hat{Y}_0) = \langle s_2\rangle $ and we restrict our attention to this case.  

Now, it is easy to show that when $m \in \{\langle s_1\rangle,\langle s_1\rangle+\frac{2}{n},\dots,\langle s_2\rangle +\frac{2}{n}\}$, $\frac{p^-}{p^+} \leq \alpha(\beta) < 1$ for $n$ large enough.
This allows us to compare our chain to the following birth-death Markov chain $(N_i)_{i=0}^{\infty}$ over the state space $\mathcal{X} := \{\langle s_1\rangle,\langle s_1\rangle+\frac{2}{n},\dots,\langle s_2\rangle+\frac{2}{n}\}$ with $N_0 = \langle s_2\rangle$. Denote the transition matrix of the birth-death chain by $\Gamma$ and let $r = |\mathcal{X}|$. By our definition of $s_2$ and $s_1$, it is clear that $r \geq c(\beta)n$ for some constant $c(\beta) > 0$. We pick $n$ large enough so that $r \geq 2$. Define the transition probabilities for $m \in \mathcal{X}$, $m\neq \langle s_1\rangle$  and $m \neq \langle s_2\rangle + \frac{2}{n}$ as follows:
$$\Gamma\left(m, m+\frac{2}{n}\right) = \Gamma\left(\langle s_1\rangle,\langle s_1\rangle +\frac{2}{n}\right) = \frac{1}{1+\alpha}$$
$$\Gamma\left(m,m-\frac{2}{n}\right) = \Gamma\left(\langle s_2\rangle+\frac{2}{n}, \langle s_2\rangle\right) = \frac{\alpha}{1+\alpha}$$
$$\Gamma\left(\langle s_1\rangle,\langle s_1\rangle\right) = \frac{\alpha}{1+ \alpha}$$
$$\Gamma\left(\langle s_2\rangle+\frac{2}{n},\langle s_2\rangle+\frac{2}{n}\right) = \frac{1}{1+\alpha}\,.$$
We couple the walk $\Gamma$ with the magnetization chain as follows:
\begin{enumerate}
\item Let $m_t$ be the magnetization chain started such that $m_0 = \langle s_2\rangle $. Let $t_i$ be the $i$th time such that $m_{t_i} \neq m_{t_{i+1}}$ and $m_{t_i} \in \{\langle s_1\rangle,\langle s_1\rangle+\frac{2}{n},\dots,\langle s_2\rangle+\frac{2}{n}\}$. Clearly, $t_i \geq i$ and the set $\{t_i :i \geq 0\}$ is infinite a.s.
\item Let $N_{i+1} = N_i -1$ if $m_{t_i} = m_{t_{i+1}} - \frac{2}{n}$.
\item If $m_{t_i} = m_{t_{i+1}} + \frac{2}{n}$, then
\begin{equation*}
N_{i+1} = \begin{cases}
N_i -1 &\quad \text{w.p. $\gamma(m_{t_i})$}\\
N_i +1 &\quad \text{w.p. $1-\gamma(m_{t_i})$}\,,
\end{cases}
\end{equation*}
where $\gamma(m_{t_i}) = \frac{p^+(m_{t_i})+ p^-(m_{t_i})}{p^+(m_{t_i})}\left(\frac{\alpha}{1+\alpha}-\frac{p^{-}(m_{t_i})}{p^+(m_{t_i})+ p^-(m_{t_i})}\right)$.
\item
The coupling above ensures that $N_i \leq m_{t_i}$ a.s whenever $t_i \leq \tau_1$.
\end{enumerate}

Let $\tau_1^{\prime} := \inf\{t: N_t = \langle s_1\rangle\}$. It follows from the coupling argument above that for any $K \in \mathbb{N}$
\begin{equation}
\mathbb{P}(\tau_1 \leq K) \leq \mathbb{P}(\tau_1^{\prime} \leq K)\,.
\label{eq:hitting_time_inequality_1}
\end{equation}

For every $k\in \mathbb{N}$, define hitting time $T_k$ as the time taken by the birth-death chain $(N_i)$ to hit the set $\{\langle s_1\rangle,\langle s_2\rangle + \frac{2}{n}\}$ for the $k$th time. By irreducibility of this Markov chain, it is clear that $T_k < \infty$ a.s. for every $k$.
Let $A_i := \{N_{T_i} = \langle s_1\rangle \}$ and  $\eta := \inf \{i: N_{T_i} = \langle s_1\rangle\}$. Clearly, $\tau_1^{\prime} \geq \eta$ a.s. Therefore, 
\begin{equation}
\mathbb{P}(\tau_1 \leq K)\leq \mathbb{P}(\tau_1^{\prime} \leq K) \leq \mathbb{P}(\eta \leq K)
\label{eq:hitting_time_inequality_2}
\end{equation}

\begin{lemma}\label{number_hit_bound}
$\mathbb{P}(\tau_1 \leq K) \leq K^2 \mathbb{P}(A_1)$.
\end{lemma}
\begin{proof}
 From Equation~\eqref{eq:hitting_time_inequality_2},
$$\mathbb{P}(\tau_1 \leq K)\leq \mathbb{P}(\eta \leq K)\,.$$
From the definition of $\eta$ and $A_i$,we have $$\{ \eta \leq K\} = \cup_{i=1}^{K}A_i\,.$$
Therefore, 
\begin{equation} 
\mathbb{P}(\tau_1 \leq K) \leq \mathbb{P}(\cup_{i=1}^K A_i) \leq \sum_{i=1}^{K} \mathbb{P}(A_i)\,.
 \label{union_bound_eta}
  \end{equation}
  
 We first prove by induction that
 \begin{equation}
 \mathbb{P}(A_i) \leq i\mathbb{P}(A_1)\,.\label{eq:hitting_induction}
 \end{equation} This is trivially true for $i=1$. Suppose it is true for some $i$. Then, 
\begin{align*}
\mathbb{P}(A_{i+1}) &= \mathbb{P}(A_{i+1}| A_{i})\mathbb{P}(A_i) + \mathbb{P}(A_{i+1}|A^{c}_{i})\mathbb{P}(A^c_i) \\
&\leq \mathbb{P}(A_i) + \mathbb{P}(A_{i+1}|A^{c}_{i}) \\
&\leq \mathbb{P}(A_i) + \mathbb{P}(A_1)\\
&\leq (i+1)\mathbb{P}(A_1)\,,
\end{align*}
completing the induction.  Here we have used the fact that conditioned on the event $A^{c}_{i}$, the walk after $T_i$ is the same as the walk starting from $\langle s_2\rangle + \frac{2}{n}$ whereas the original walk at time $t=0$ starts from $\langle s_2\rangle$.  Therefore, $\mathbb{P}\left(N_{T_{i+1}} = \langle s_1\rangle| A_i^c\right) \leq \mathbb{P}\left(N_{T_1} = \langle s_1\rangle\right)$, which is the same as  $\mathbb{P}(A_{i+1}|A^{c}_{i}) \leq \mathbb{P}(A_1)$.
Combining Equation~\eqref{eq:hitting_induction} with Equation~\eqref{union_bound_eta}, we arrive at the conclusion of Lemma~\ref{number_hit_bound}.
\end{proof}
 For the sake of convenience, we rename the states of $\mathcal{X}$ to be elements in $\{0,..,r-1\}$ with the same ordering (i.e, $\langle s_1\rangle \to 0$, $\langle s_1\rangle+\frac{2}{n} \to 1,\dots,\langle s_2\rangle+\frac{2}{n} \to r-1$). Let $p = \frac{1}{1+\alpha}$ denote the probability of moving from state $m$ to $m+1$ and $1-p$ denote the probability of moving from $m$ to $m-1$. Let $P_m$ be the probability that the Markov chain starting at state $m$ hits $r-1$ before it hits $0$. The following lemma is a classic result about biased Gambler's ruin Markov chain. We assume that $n$ is large enough so that $r \geq 2$.
 
\begin{lemma}\label{hitting_bound}
$1 - P_{r-2} = \mathbb{P}(A_1) \leq (\frac{1-p}{p})^{r-2} = \alpha^{r-2}$.
\end{lemma}
\begin{proof}
We have the following set of recursion equations:
$P_0 = 0$, $P_{r-1} = 1$ and for all $0 <i < r-1 $, $P_m = pP_{m+1} + (1-p)P_{m-1}$.
One can check that the unique solution to this set of equations is $$P_m = \frac{(\frac{1-p}{p})^m -1}{(\frac{1-p}{p})^{r-1}-1} = \frac{\alpha^m - 1}{\alpha^{r-1} - 1}\,.
$$
By definition of the event $A_1$, 
\begin{align*}
1 - P_{r-2} &= \mathbb{P}(A_1) \\
&= \alpha^{r-2}\frac{1-\alpha}{1- \alpha^{r-1}} \\
&\leq  \alpha^{r-2}\,. &\qedhere
\end{align*}
\end{proof}

From Lemmas~\ref{number_hit_bound} and~\ref{hitting_bound}, we conclude that
$$\mathbb{P}(\tau_1 \leq K) \leq K^2 \alpha^{r-2} \,.$$

As shown above, $r -2 \geq c(\beta )n$ for constant $c(\beta) > 0$ and $\alpha = \alpha(\beta) < 1$. Therefore, for some constant $c_1(\beta) > 0$,
$$\mathbb{P}(\tau_1 \leq K) \leq K^2 \exp{(-c_1(\beta)n)}\,,$$
and this completes the proof.

\end{document}